\documentclass[11pt]{article}
\usepackage{eurosym}
\usepackage{amssymb}
\usepackage{amsfonts}
\usepackage{amsmath}
\usepackage{graphicx}
\usepackage{hyperref}
\usepackage{subcaption}
\usepackage{tikz}
\usepackage{float}
\usepackage{cite}
\usepackage{float}
\newtheorem{theorem}{Theorem}

\newtheorem{lemma}[theorem]{Lemma}

\newenvironment{proof}[1][Proof]{\textbf{#1.} }{\ \rule{0.5em}{0.5em}}
\begin{document}

\author{ Pelin G. Geredeli \thanks{%
email address: peling@iastate.edu.} \\
Department of Mathematics\\
Iowa State University, Ames-IA, USA}
\title{Asymptotic Stability of a Compressible Oseen-Structure Interaction via a Pointwise
Resolvent Criterion }
\maketitle

\begin{abstract}
In this study, we consider a linearized compressible flow structure
interaction PDE model for which the interaction interface is under the
effect of material derivative term. While the linearization takes place
around a constant pressure and density components in structure equation, the
flow linearization is taken with respect to a non-zero, fixed, variable
ambient vector field. This process produces extra ``convective derivative" and ``material derivative" terms which causes the coupled system to be nondissipative.

We analyze the long time dynamics in the sense of asymptotic (strong)
stability in an invariant subspace (one dimensional less) of the entire
state space where the continuous semigroup is ``\textit{uniformly bounded}".
For this, we appeal to the pointwise resolvent condition introduced in \cite%
{CT} which avoids many technical complexity and provides a very clean, short
and easy-to-follow proof.

\vskip.3cm \noindent \textbf{Key terms:} Flow-structure interaction,
compressible flows, stability, resolvent, uniformly bounded semigroup,
material derivative
\end{abstract}


\section{ Introduction}

\hspace{0.4cm} The mathematical analysis of fluid structure interaction
(FSI) problems constitutes a broad area of research with applications in
aeroelasticity, biomechanics and fluid dynamics \cite{clark, dvorak,
george1, george2, pelin-george, agw, agm, material, T1, T2, ALT, lorena,
clw, cr, igor, igor2, p1, sima, canic}. Such interactive dynamics between
flow/fluid and a plate (or shell) are mathematically realized by coupled PDE
systems with compressible flow and elastic plate components. The analysis of
these PDE systems is considered from many points of view \cite{pelin-george,
agw,material,igor,p1,sima}.

In this work, we consider a linearized flow-structure PDE model with respect
to some reference state which results in the appearance of an arbitrary
spatial flow field. In contrast to the incompressible case, having a
compressible flow component presents a great many difficulties due to the
increase in the number of unknown variables. By the nature and physics of
compressible flows, density can change by pressure forces and a new set of
governing equations are necessarily derived along with the equations for the
conservation of mass and momentum. These equations should be valid for the
flows (compressible) whose range of Mach number is 
\begin{equation*}
MachNumber=M=\frac{\text{velocity}}{\text{local speed of sound}}>0.3.
\end{equation*}%
The cases $M<0.3$ and $0.3<M<0.8$ are subsonic/incompressible and
subsonic/compressible regimes, respectively. Compressible flows can be
either transonic $(0.8<M<1.2)$ or supersonic $(1.2<M<3.0)$. In supersonic
flows, pressure effects are only transported downstream; the upstream flow
is not affected by conditions downstream.

Our principle aim is to consider the long-time behavior of the corresponding
coupled FSI system with a focus of (asymptotic) strong stability properties
of the $C_{0}-$semigroup generated by the solution. This asymptotic decay
for solutions of the compressible flow-structure
PDE model will be stated within the context of the associated semigroup
formulation and ``frequency domain" approach.

\vspace{0.5cm} \noindent \textbf{The FSI Geometry}\newline

Let the flow domain $\mathcal{O} \subset \mathbb{R}^{3}$ with Lipschitz
boundary $\partial \mathcal{O}$. We assume that $\partial \mathcal{O}=%
\overline{S}\cup \overline{\Omega } $, with $S\cap \Omega =\emptyset $, and
the (structure) domain $\Omega \subset \mathbb{R}^{3}$ is a \emph{flat}
portion of $\partial \mathcal{O}$ with $C^2-$ boundary. In particular, $%
\partial \mathcal{O}$ has the following specific configuration: 
\begin{equation}
\Omega \subset \left\{ x=(x_{1,}x_{2},0)\right\} \,\text{\ and \ surface }%
S\subset \left\{ x=(x_{1,}x_{2},x_{3}):x_{3}\leq 0\right\} \,.  \label{geo}
\end{equation}

Additionally, the flow domain $\mathcal{O}$ should be curvilinear polyhedral
domain which satisfies the following conditions:

\begin{itemize}
\item Each corner of the boundary $\partial \mathcal{O}$ -if any- is
diffeomorphic to a convex cone, \label{g1}

\item Each point on an edge of the boundary $\partial \mathcal{O}$ is
diffeomorphic to a wedge with opening $<\pi.$ \label{g2}
\end{itemize}

We note that these additional conditions on the flow domain $\mathcal{O}$
are necessary for the application of some elliptic regularity results for
solutions of second order boundary value problems on corner domains \cite%
{dauge, Dauge_3}. We denote the unit outward normal vector to $\partial 
\mathcal{O}$ by $\ \mathbf{n}(\mathbf{x})$ where $\mathbf{n|}_{\Omega
}=[0,0,1]$, and the unit outward normal vector to $\partial \Omega$ by $%
\mathbf{\nu}\mathbf{(x)}$. Some examples of geometries can be seen in Figure
1.\newline

\begin{figure}[]
\begin{subfigure}[H]{0.3\linewidth}
\centering
\begin{tikzpicture}[scale=0.9]
\draw[left color=black!10,right color=black!20,middle
color=black!50, ultra thick] (-2,0,0) to [out=0, in=180] (2,0,0)
to [out=270, in = 0] (0,-3,0) to [out=180, in =270] (-2,0,0);

\draw [fill=black!60, ultra thick] (-2,0,0) to [out=80,
in=205](-1.214,.607,0) to [out=25, in=180](0.2,.8,0) to [out=0,
in=155] (1.614,.507,0) to [out=335, in=100](2,0,0) to [out=270,
in=25] (1.214,-.507,0) to [out=205, in=0](-0.2,-.8,0) [out=180,
in=335] to (-1.614,-.607,0) to [out=155, in=260] (-2,0,0);

\draw [dashed, thin] (-1.7,-1.7,0) to [out=80, in=225](-.6,-1.3,0)
to [out=25, in=180](0.35,-1.1,0) to [out=0, in=155] (1.3,-1.4,0)
to [out=335, in=100](1.65,-1.7,0) to [out=270, in=25] (0.9,-2.0,0)
to [out=205, in=0](-0.2,-2.2,0) [out=180, in=335] to (-1.514,-2.0)
to [out=155, in=290] (-1.65,-1.7,0);

\node at (0.2,0.1,0) {{\LARGE$\Omega$}};

\node at (1.95,-1.5,0) {{\LARGE $S$}};

\node at (-0.3,-1.6,0) {{\LARGE $\mathcal{O}$}};
\end{tikzpicture}

 \end{subfigure}
\hfill 
\begin{subfigure}[H]{0.28\linewidth}
\centering
\includegraphics[width=\linewidth]{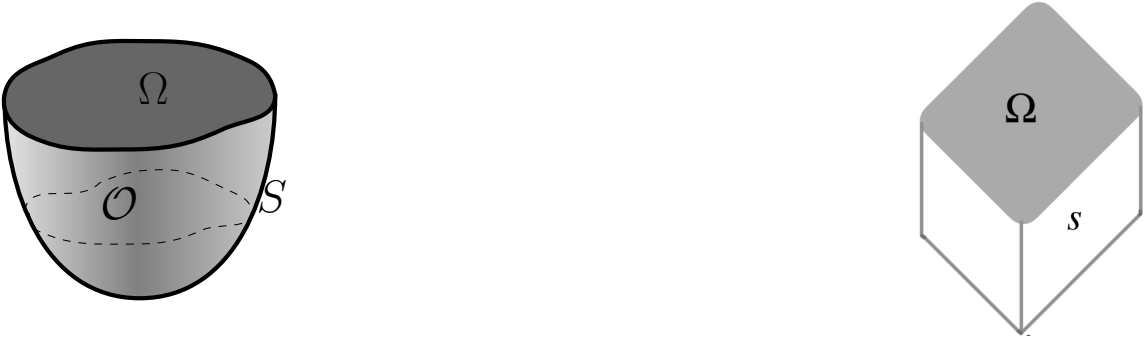}
\end{subfigure}
\caption{Polyhedral Flow-Structure Geometries }
\end{figure}

\noindent \textbf{Linearization and the PDE Model.}\newline

In what follows we provide some information about the linearization process
and the PDE description of the compressible flow-structure interaction
system under consideration. Firstly, we note that since the linear flow
problem here is already of great technical complexity and mathematical
challenge we assume that the pressure is a linear function of the density; $%
p(x,t)=C\rho (x,t),$ as is typically seen in the compressible flow
literature, and it is chosen as a primary variable to solve. For further and
detailed explanations of the physical background concerning the relationship
between pressure and density, the reader is referred to \cite{sima, material}%
.

The linearization takes place around an equilibrium point of the form $%
\left\{ p_{\ast },\mathbf{U},\varrho _{\ast }\right\} $ where the pressure
and density components ${p_{\ast },\varrho _{\ast }}$ are assumed to be
scalars (for  simplicity, assume $p_{\ast }=\varrho _{\ast }=1$), and a
generally non-zero, fixed, ambient vector field $\mathbf{U}:\mathcal{O}%
\rightarrow \mathbb{R}^{3}$ 
\begin{equation*}
\mathbf{U}%
(x_{1},x_{2},x_{3})=[U_{1}(x_{1},x_{2},x_{3}),U_{2}(x_{1},x_{2},x_{3}),U_{3}(x_{1},x_{2},x_{3})].
\end{equation*}%
At this point, we emphasize that flow linearization is taken with respect to
some inhomogeneous compressible Navier-Stokes system; thus, $%
\mathbf{U}$ does not need to be divergence free, generally.

Now, with respect to the above linearization, the small perturbations give
the following physical equations by generalizing the forcing functions: 
\begin{equation*}
(\partial _{t}+\mathbf{U}\cdot \nabla )p+\text{div}(u)+(\text{div}~\mathbf{U}%
)p=f(\mathbf{x})\quad {\ \ in\ \ }\mathcal{O}\times \mathbb{R}_{+},
\end{equation*}%
\begin{equation*}
(\partial _{t}+\mathbf{U}\cdot \nabla )u-\nu \Delta u-(\nu +\lambda )\nabla 
\text{div}u+\nabla p+\nabla \mathbf{U}\cdot u+(\mathbf{U}\cdot \nabla 
\mathbf{U})p=\mathbf{F(\mathbf{x})}\quad {\ \ in\ }~\mathcal{O}\times 
\mathbb{R}_{+}.
\end{equation*}%
(For further discussion, see also \cite{igor, material}.) When we delete
some of the non-critical lower order and the benign inhomogeneous terms in
the above equations, this linearization gives rise to the following system
of equations, in solution variables $u(x_{1},x_{2},x_{3},t)$ (flow
velocity), $p(x_{1},x_{2},x_{3},t)$ (pressure), $w_{1}(x_{1},x_{2},t)$
(elastic plate displacement) and $w_{2}(x_{1},x_{2},t)$ (elastic plate
velocity): 
\begin{align}
& \left\{ 
\begin{array}{l}
p_{t}+\mathbf{U}\cdot \nabla p+\text{div}~u\mathbf{+}\text{div}(\mathbf{U)}%
p=0~\text{ in }~\mathcal{O}\times (0,\infty ) \\ 
u_{t}+\mathbf{U}\cdot \nabla u+u\cdot \nabla \mathbf{U}-\text{div}\sigma
(u)+\eta u+\nabla p=0~\text{ in }~\mathcal{O}\times (0,\infty ) \\ 
(\sigma (u)\mathbf{n}-p\mathbf{n})\cdot \boldsymbol{\tau }=0~\text{ on }%
~\partial \mathcal{O}\times (0,\infty ) \\ 
u\cdot \mathbf{n}=0~\text{ on }~S\times (0,\infty ) \\ 
u\cdot \mathbf{n}=w_{2}+\mathbf{U}\cdot \nabla w_{1}\text{ \ \ on }~\Omega
\times (0,\infty )\text{ }%
\end{array}%
\right.   \label{1} \\
& \left\{ 
\begin{array}{l}
w_{1_{t}}-w_{2}-\mathbf{U}\cdot \nabla w_{1}=0\text{ \ \ on }~\Omega \times
(0,\infty ) \\ 
w_{2_{t}}+\Delta ^{2}w_{1}+\left[ 2\nu \partial _{x_{3}}(u)_{3}+\lambda 
\text{div}(u)-p\right] _{\Omega }=0~\text{ on }~\Omega \times (0,\infty ) \\ 
w_{1}=\frac{\partial w_{1}}{\partial \nu }=0~\text{ on }~\partial \Omega
\times (0,\infty )%
\end{array}%
\right.   \label{2} \\
& 
\begin{array}{c}
\left[ p(0),u(0),w_{1}(0),w_{2}(0)\right] =\left[ \overline{p},\overline{u},%
\overline{w_{1}},\overline{w_{2}}\right] \in \mathcal{H}_{0}.%
\end{array}
\label{3}
\end{align}%
Here, $\mathcal{H}_{0}$ is given as follows: 
\begin{equation}
\mathcal{H}_{0}=\mathcal{\{}[p_{0},u_{0},w_{1},w_{2}]\in \mathcal{H}%
:\int\limits_{\mathcal{O}}p_{0}d\mathcal{O}+\int\limits_{\Omega
}w_{1}d\Omega =0\mathcal{\}},\text{\ \ } \label{null-ort}
\end{equation}%
where 
\begin{equation}
\mathcal{H}\equiv L^{2}(\mathcal{O})\times \mathbf{L}^{2}(\mathcal{O})\times
H_{0}^{2}(\Omega )\times L^{2}(\Omega )  \label{stand}
\end{equation}%
is the associated finite energy (Hilbert) space, topologized by the standard
inner product: 
\begin{equation}
(\mathbf{y}_{1},\mathbf{y}_{2})_{\mathcal{H}}=(p_{1},p_{2})_{L^{2}(\mathcal{O%
})}+(u_{1},u_{2})_{\mathbf{L}^{2}(\mathcal{O})}+(\Delta w_{1},\Delta
w_{2})_{L^{2}(\Omega )}+(v_{1},v_{2})_{L^{2}(\Omega )}  \label{stand}
\end{equation}%
for any $\mathbf{y}_{i}=(p_{i},u_{i},w_{i},v_{i})\in \mathcal{H},~i=1,2.$%
\newline

It was shown in \cite{pelin-george} and \cite{p1} that $\mathcal{H}_{0}^{\bot }$ is the null
space of the flow structure semigroup generator closely associated with
(2)-(4). It will be shown below that solutions of (2)-(4), with initial data
drawn from $\mathcal{H}_{0}$, decay asymptotically to the zero state. Also,
the terms $\mathbf{U}\cdot \nabla u+u\cdot \nabla \mathbf{U}$ constitute the
so-called Oseen (linear) approximation of the Navier-Stokes equations \cite{Yudo}.

The quantity $\eta >0$ represents a drag force of the domain on the viscous
flow. In addition, the quantity $\mathbf{\tau }$ in (\ref{1}) is in the
space $TH^{1/2}(\partial \mathcal{O)}$ of tangential vector fields of
Sobolev index 1/2; that is,%
\begin{equation}
\mathbf{\tau }\in TH^{1/2}(\partial \mathcal{O)=}\{\mathbf{v}\in \mathbf{H}^{%
\frac{1}{2}}(\partial \mathcal{O}):\mathbf{v}|_{\partial \mathcal{O}}\cdot 
\mathbf{n}=0~\text{ on }~\partial \mathcal{O}\}.
\end{equation}%
(See e.g., p.846 of \cite{buffa2}.) In addition, we take ambient field $%
\mathbf{U}\in \mathbf{V}_{0}\cap W$ where 
\begin{equation}
\mathbf{V}_{0}=\{\mathbf{v}\in \mathbf{H}^{1}(\mathcal{O})~:~\left. \mathbf{v%
}\right\vert _{\partial \mathcal{O}}\cdot \mathbf{n}=0~\text{ on }~\partial 
\mathcal{O}\},  \label{V_0}
\end{equation}%
\begin{equation}
W=\{v\in \mathbf{H}^{1}(\mathcal{O}):v\in L^{\infty }(\mathcal{O}),\text{ \
\ }div(v)\in L^{\infty }(\mathcal{O}),\text{ \ \ and \ \ }v|_{\Omega }\in
C^{2}(\overline{\Omega })\}  \label{W}
\end{equation}

(The vanishing of the boundary for ambient fields is a standard assumption
in compressible flow literature; see \cite{dV},\cite{valli},\cite{decay},%
\cite{spectral}.) Moreover, the \textit{stress and strain tensors} in the
flow PDE component of (\ref{1})-(\ref{3}) are defined respectively as 
\begin{equation*}
\sigma (\mathbf{\mu })=2\nu \epsilon (\mathbf{\mu })+\lambda \lbrack
I_{3}\cdot \epsilon (\mathbf{\mu })]I_{3};\text{ \ }\epsilon _{ij}(\mathbf{%
\mu })=\dfrac{1}{2}\left( \frac{\partial \mathbf{\mu }_{j}}{\partial x_{i}}+%
\frac{\partial \mathbf{\mu }_{i}}{\partial x_{j}}\right) \text{, \ }1\leq
i,j\leq 3,
\end{equation*}%
where \textit{Lam\'{e} Coefficients }$\lambda \geq 0$ and $\nu >0$.

Here, we impose the so called \emph{impermeability condition} on $\Omega $;
namely, we assume that no fluid passes through the elastic portion of the
boundary during deflection \cite{bolotin,dowell1}. Also, note that the FSI
problem under consideration has a \emph{material derivative} term on the
deflected interaction surface which computes the time rate of change of any
quantity such as temperature or velocity (and hence also acceleration) for a
portion of a material in motion. For further details and the physical
explanation of the material derivative boundary conditions, see \cite%
{material, sima}.

\section{Previous Considerations}

\hspace{0.4cm} The long-time behavior, in particular the stability
properties of incompressible/compressible fluid/flow structure interaction
(FSI) systems have been a popular topic treated by many authors at different
levels \cite{clark, dvorak, george1, george2, pelin-george, agm, ALT, clw,
cr, igor2, p1}. The PDE systems under consideration are generally quite
complicated, due to their unbounded hyperbolic-parabolic coupling
mechanisms, both being intrinsic to the underlying physics. 

In contrast to the large body of literature on incompressible FSI \cite%
{clark, dvorak, george1, george2, agm, T1, T2, ALT, lorena, clw, cr, igor2,
canic}, existing work on compressible flows which interact with elastic
solids is relatively limited due to the inherent mathematical challenges
presented by the extra density (pressure) variable. In analyzing these
compressible FSI PDE models there are the following challenges: (i) Given
that the flow-plate variables are coupled via boundary interfaces, the FSI
geometry is inherently nonsmooth, \cite{igor, dV}. Although such geometries
are physically relevant, these domains also give rise to regularity issues
for solutions; (ii) The linearization of the compressible FSI system under
consideration takes place around a rest state which includes a generally
nonzero (not constant) ambient flow field $\mathbf{U}$, and the said
pressure PDE will contain a ``convective derivative"
term $\mathbf{U}\nabla p$ which is strictly above the level of finite
energy; (iii) The boundary conditions which couple flow and structure
contain nondissipative and unbounded terms which complicate the PDE analysis.

A preliminary linearized compressible FSI model with $\mathbf{U}\equiv 0$
(and so \textit{without} material derivative) was derived by I. Chueshov 
\cite{igor}. In this work, the author showed the wellposedness and the
existence of global attractors in the case that the structure equation has a
von Karman nonlinearity. However, in this pioneering work, the author noted
that his methods (Galerkin approach and Lyapunov functionals) to show the
wellposedness and long-term behavior of the corresponding system would not
accommodate the case of interest; namely linearization about $\mathbf{U}\neq
0$.

The main reason for the difficulty is the need to control the ``convective derivative" term $\mathbf{U}\nabla p$ which requires a
decomposition of the fluid and pressure solution pair $\{u,p\}$ with an
eventual appeal to some elliptic regularity results on nonsmooth domains.
Subsequently, the suggested model was analyzed in \cite{agw} via a semigroup
formulation and a wellposedness result with additional interior terms
associated to the $\mathbf{U}\neq 0$ under a pure velocity matching
condition at the interface was obtained. Then in \cite{material}, the
authors re-visited the same problem after a careful derivation of
fluid-structure interface conditions written in terms of ``material
derivative" $(\partial t+\mathbf{U}\cdot \nabla )w.$ This material
derivative computes the time rate of change of any quantity such as
temperature or velocity (and hence also acceleration) for a portion of a
material in motion. However, since the material derivative term $\mathbf{U}%
\nabla w$ is unbounded and nondissipative, it adds an additional challenge
to the analysis. In \cite{material}, this lack of boundedness and
dissipativity were ultimately overcome by re-topologizing the finite energy
space, and so the desired wellposedness result was obtained by an
appropriate semigroup formulation. At this point, we should note that in
both papers \cite{agw, material}, the obtained semigroup is unfortunately
NOT uniformly bounded (in time) which prevents us to seek long term behavior
of the solution to these corresponding problems.

With a view of looking into stability properties of the flow structure PDE
models considered in \cite{agw, material}, it appeared natural to ask: Is it
possible to obtain a semigroup wellposedness and subsequently a stability
result, with the semigroup being bounded uniformly in time, at least in some
(inherently invariant) subspace of the finite energy space? This was indeed
a very important departure point in order to analyze the stability of these
FSI problems since there was not any long time behavior result in the
literature for such classes of FSI systems. Motivated by this question, an
initial result of uniform stability result for the solution to linearized
compressible FSI model (\textit{without material derivative}) in an
appropriate subspace was shown in \cite{pelin-george, p1} by using linear
perturbation theory and also some novel multipliers.

\section{Novelty}

\hspace{0.4cm} Having established the uniform decay result for the solution
to the ``\textit{material free}" compressible FSI system in \cite%
{pelin-george, p1}, our next goal was to analyze the long time behavior of
the system, when under the effect of ``material term" on the interaction
interface. In order to embark on this stability work, we firstly needed to
have a uniformly bounded (in time) semigroup, again in some (inherently
invariant) subspace of the finite energy space. This result was obtained in 
\cite{sima}. Now, in the present work, we provide a novel result on the
asymptotic (strong) stability of compressible FSI model under said ``\textit{%
material derivative}" boundary conditions. This result ascertains that
solution to the FSI model decays to zero state for all initial data taken
from a subspace $\mathcal{H}_{0}$ which is ``almost" the entire phase space $%
\mathcal{H}$. This is, \textit{to the knowledge of the author}, the first
such result obtained with ``\textit{material derivative}" BC.

By way of obtaining the aforesaid asymptotic (strong) stability result, we
operate in the frequency domain which requires us
to deal with some static equations, analogous to (2)-(3), and the resolvent
of the generator of the dynamical system. While we have already been aware
of the need to control the ``convective derivative"
term $\mathbf{U}\nabla p$ from our previous works, in this current
manuscript the primary challenge is to have tight control of the material derivative
term $\mathbf{U}\nabla w$ in the coupling condition at the interface since
it destroys the dissipative nature of the dynamics. In this regard, the main
challenges associated with the analysis and the novelties are as follows:%
\newline

\noindent \textbf{(i)} \textit{Inner product adjustment for dissipativity:}
In order to establish a long time behavior result for the given FSI system,
we expect that the energy of this system is decreasing. However, the
presence of the problematic convective and material derivative terms $%
\mathbf{U}\nabla p$ and $\mathbf{U}\nabla w$ adds a great challenge to the
long term analysis since they can not be treated as a perturbation and
moreover cause a lack of dissipativity property in the system. To deal with
these issues and get the necessary dissipativity estimate we change the
inner-product structure of the state space and we re-topologize the
inherently invariant subspace with an inner product
which is equivalent to the standard inner product defined for the entire
state space. In this construction, we make use of a multiplier which
exploits the characterization of this invariant subspace and the Dirichlet
map that extends boundary data defined on the interaction interface to a
harmonic function in the flow domain. See Theorem \ref{wp} below and \cite{sima}.\newline

\noindent \textbf{(ii)} \textit{Use of pointwise resolvent criterion for
stability:} Since the domain of the associated flow-structure generator is
not compactly embedded into finite energy space $\mathcal{H}$, (see
(\ref{AAA})-(\ref{feedbackB}) below) it would seem natural to appeal to well known strong
stability result introduced by Arend-Batty \cite{arendt}. However, this does
not seem practicable for our model since the spectral analysis of the
semigroup generator will require a great amount of technicalities. For
example, to analyze the residual spectrum of the generator, we need to check
that none of the points on the imaginary axis are the eigenvalues of the
adjoint operator of this generator. Unfortunately, it is not going to be
straightforward to get this result since the adjoint, itself, is given as
the sum of three complicated matrix (see Lemma \ref{adj}). Instead, we will
use the very interesting pointwise resolvent criterion developed by
Chill and Tomilov \cite{CT, Tomilov_2} for the stability of bounded semigroups which
perfectly works for our model (See also \cite{B-L} which gives a preliminary version of the resolvent criterion for strong stability.) Hence we provide a very clean, short and
easy-to-follow proof which does not touch the --very complicated-- adjoint
operator and the challenges it might cause.

\section{Preliminaries}

\ Throughout, for a given domain $D$, the norm of corresponding space $%
L^{2}(D)$ will be denoted as $||\cdot ||_{D}$ (or simply $||\cdot ||$ when
the context is clear). Inner products in $L^{2}(\mathcal{O})$ or $\mathbf{L}%
^{2}(\mathcal{O})$ will be denoted by $(\cdot ,\cdot )_{\mathcal{O}}$,
whereas inner products $L^{2}(\partial \mathcal{O})$ will be written as $%
\langle \cdot ,\cdot \rangle _{\partial \mathcal{O}}$. We will also denote
pertinent duality pairings as $\left\langle \cdot ,\cdot \right\rangle
_{X\times X^{\prime }}$ for a given Hilbert space $X$. The space $H^{s}(D)$
will denote the Sobolev space of order $s$, defined on a domain $D$; $%
H_{0}^{s}(D)$ will denote the closure of $C_{0}^{\infty }(D)$ in the $%
H^{s}(D)$-norm $\Vert \cdot \Vert _{H^{s}(D)}$. We make use of the standard
notation for the boundary trace of functions defined on $\mathcal{O}$, which
are sufficently smooth: i.e., for a scalar function $\phi \in H^{s}(\mathcal{%
O})$, $\frac{1}{2}<s<\frac{3}{2}$, $\gamma (\phi )=\phi \big|_{\partial 
\mathcal{O}},$ which is a well-defined and surjective mapping on this range
of $s$, owing to the Sobolev Trace Theorem on Lipschitz domains (see e.g., 
\cite{necas}, or Theorem 3.38 of \cite{Mc}). Also, $C>0$ will denote a generic constant.\newline

\noindent In order to reduce the PDE system (\ref{1})-(\ref{3}) to an ODE
setting in the Hilbert space $\mathcal{H},$ define the operators

\begin{equation}
\mathcal{A}=\left[ 
\begin{array}{cccc}
-\mathbf{U}\mathbb{\cdot }\nabla (\cdot ) & -\text{div}(\cdot ) & 0 & 0 \\ 
-\mathbb{\nabla (\cdot )} & \text{div}\sigma (\cdot )-\eta I-\mathbf{U}%
\mathbb{\cdot \nabla (\cdot )-}\nabla \mathbf{U}\cdot (\cdot ) & 0 & 0 \\ 
0 & 0 & 0 & I \\ 
\left. \left[ \cdot \right] \right\vert _{\Omega } & -\left[ 2\nu \partial
_{x_{3}}(\cdot )_{3}+\lambda \text{div}(\cdot )\right] _{\Omega } & -\Delta
^{2} & 0%
\end{array}%
\right] ;  \label{AAA}
\end{equation}%
and 
\begin{equation}
B=\left[ 
\begin{array}{cccc}
-\text{div}(\mathbf{U)(\cdot )} & 0 & 0 & 0 \\ 
0 & 0 & 0 & 0 \\ 
0 & 0 & \mathbf{U}\mathbb{\cdot }\nabla (\cdot ) & 0 \\ 
0 & 0 & 0 & 0%
\end{array}%
\right]   \label{feedbackB}
\end{equation}%
\noindent where $D(\mathcal{A}+B)\subset \mathcal{H}$ is given by

\begin{equation*}
D(\mathcal{A}+B)=\{(p_{0},u_{0},w_{1},w_{2})\in L^{2}(\mathcal{O})\times 
\mathbf{H}^{1}(\mathcal{O})\times H_{0}^{2}(\Omega )\times L^{2}(\Omega )~:~%
\text{properties }(A.i)\text{--}(A.vi)~~\text{hold}\}:
\end{equation*}

\begin{enumerate}
\item[\textbf{(A.i)}] $\mathbf{U}\cdot \nabla p_{0}\in L^{2}(\mathcal{O})$

\item[\textbf{(A.ii)}] $\text{div}~\sigma (u_{0})-\nabla p_{0}\in \mathbf{L}%
^{2}(\mathcal{O})$ (So, $\left[ \sigma (u_{0})\mathbf{n}-p_{0}\mathbf{n}%
\right] _{\partial \mathcal{O}}\in \mathbf{H}^{-\frac{1}{2}}(\partial 
\mathcal{O})$)

\item[\textbf{(A.iii)}] $-\Delta ^{2}w_{1}-\left[ 2\nu \partial
_{x_{3}}(u_{0})_{3}+\lambda \text{div}(u_{0})\right] _{\Omega
}+p_{0}|_{\Omega }\in L^{2}(\Omega )$ (by elliptic regularity theory $w_1
\in H^{3}(\Omega))$

\item[\textbf{(A.iv)}] $\left( \sigma (u_{0})\mathbf{n}-p_{0}\mathbf{n}%
\right) \bot ~TH^{1/2}(\partial \mathcal{O})$. That is, 
\begin{equation*}
\left\langle \sigma (u_{0})\mathbf{n}-p_{0}\mathbf{n},\mathbf{\tau }%
\right\rangle _{\mathbf{H}^{-\frac{1}{2}}(\partial \mathcal{O})\times 
\mathbf{H}^{\frac{1}{2}}(\partial \mathcal{O})}=0\text{ \ in }\mathcal{D}%
^{\prime }(\mathcal{O})\text{\ for every }\mathbf{\tau }\in
TH^{1/2}(\partial \mathcal{O})
\end{equation*}

\item[\textbf{(A.v)}] $w_{2}+\mathbf{U}\cdot \nabla w_{1}\in
H_{0}^{2}(\Omega )$ (and so $w_{2}\in H_{0}^{1}(\Omega ))$

\item[\textbf{(A.vi)}] The flow velocity component $u_{0}=\mathbf{f}_{0}+%
\widetilde{\mathbf{f}}_{0}$, where $\mathbf{f}_{0}\in \mathbf{V}_{0}$ and $%
\widetilde{\mathbf{f}}_{0}\in \mathbf{H}^{1}(\mathcal{O})$ satisfies%
\footnote{%
The existence of an $\mathbf{H}^{1}(\mathcal{O})$-function $\widetilde{%
\mathbf{f}}_{0}$ with such a boundary trace on Lipschitz domain $\mathcal{O}$
is assured; see e.g., Theorem 3.33 of \cite{Mc}.}%
\begin{equation*}
\widetilde{\mathbf{f}}_{0}=%
\begin{cases}
0 & ~\text{ on }~S \\ 
(w_{2}+\mathbf{U}\cdot \nabla w_{1})\mathbf{n} & ~\text{ on}~\Omega%
\end{cases}%
\end{equation*}%
\noindent (and so $\mathbf{f}_{0}|_{\partial \mathcal{O}}\in
TH^{1/2}(\partial \mathcal{O})$).\newline
\end{enumerate}

\noindent Then the function $\Phi (t)=\left[ p,u,w_1,w_{2}\right] \in
C([0,T];\mathcal{H})$ that solves the problem (\ref{1})-(\ref{3}) satisfies 
\begin{eqnarray}
\dfrac{d}{dt}\Phi (t) &=&(\mathcal{A}+B)\Phi (t);  \notag \\
\Phi (0) &=&\Phi _{0}  \label{ODE}
\end{eqnarray}

\noindent We should note that the semigroup generator $(\mathcal{A}+B)$ is $%
\mathcal{H}_{0}$ - invariant (see Lemma 5 of \cite{sima}). Moreover, in \cite{sima}
we showed that the flow-structure PDE system (2)-(4), in the absence of the
term $u\cdot \nabla \mathbf{U} $ in the flow component, is associated with the generator of a $C_{0}$%
- contraction semigroup in $\mathcal{H}_{0}$. Therefore, by a bounded
perturbation argument, $(\mathcal{A}+B)$ generates a $C_{0}$- semigroup in $%
\mathcal{H}_{0},$ which, however, \underline{might} \underline{not} be a
contraction. Since the main goal of this manuscript is to show the
asymptotic (strong) stability of the PDE system (\ref{1})-(\ref{3}) for
every initial data in the subspace $\mathcal{H}_{0}$, having a uniformly
bounded semigroup will be our key point to embark on proving
our main result, Theorem \ref{St}. For the sake of simplicity, we use the
notation 
\begin{equation*}
(\mathcal{A}+B)|_{\mathcal{H}_{0}}=(\mathcal{A}+B).
\end{equation*}%
In order to analyze the long time behavior of the solution to (\ref{1})-(\ref%
{3}) in the reduced space $\mathcal{H}_{0}$ (defined in (\ref{null-ort})),
we have to guarantee that the dynamical system (\ref{1})-(\ref{3}) is
dissipative. But, unfortunately, the presence of the generally nonzero
ambient vector field $\mathbf{U}$ and the material derivative term boundary
conditions in the model cause the lack of dissipativity of the generator $(%
\mathcal{A}+B)$ (with respect to the standard inner product given in (\ref%
{stand}).) Hence, to obtain a necessary dissipativity estimate for each
solution variable, we re-topologize the phase space $\mathcal{H}$ with a new
inner product to be used in $\mathcal{H}_{0}$ and equivalent to the natural
inner product given in (\ref{stand}). (See \cite{sima}) For the readers
convenience, we provide the details here as well:\newline

With the above notation let us take $\varphi =\left[ p_{0},u_{0},w_{1},w_{2}%
\right] \in \mathcal{H}_{0},$ $\widetilde{\varphi }=\left[ \widetilde{p}_{0},%
\widetilde{u}_{0},\widetilde{w}_{1},\widetilde{w}_{2}\right] \in \mathcal{H}%
_{0}.$ Then the new inner product is given as 
\begin{equation*}
((\varphi ,\widetilde{\varphi }))_{\mathcal{H}_{0}}=(p_{0},\widetilde{p}%
_{0})_{\mathcal{O}}+(u_{0}-\alpha D(g\cdot \nabla w_{1})e_{3}+\xi \nabla
\psi (p_{0},w_{1}),\widetilde{u}_{0}-\alpha D(g\cdot \nabla \widetilde{w}%
_{1})e_{3}+\xi \nabla \psi (\widetilde{p}_{0},\widetilde{w}_{1}))_{\mathcal{O%
}}
\end{equation*}%
\begin{equation}
+(\Delta w_{1},\Delta \widetilde{w}_{1})_{\Omega }+(w_{2}+h_{\alpha }\cdot
\nabla w_{1}+\xi w_{1},\widetilde{w}_{2}+h_{\alpha }\cdot \nabla \widetilde{w%
}_{1}+\xi \widetilde{w}_{1})_{\Omega } \label{innpr},
\end{equation}%
and in turn the norm 
\begin{eqnarray}
\left\Vert \left\vert \varphi \right\vert \right\Vert _{\mathcal{H}%
_{0}}^{2} = \left( \left( \varphi ,\varphi \right) \right) _{\mathcal{H}_{0}}&=&\left\Vert p_{0}\right\Vert _{\mathcal{O}}^{2}+\left\Vert u_{0}-\alpha
D(g\cdot \nabla w_{1})e_{3}+\xi \nabla \psi (p_{0},w_{1})\right\Vert _{%
\mathcal{O}}^{2} \notag \\
&+&\left\Vert \Delta w_{1}\right\Vert _{\Omega
}^{2}+\left\Vert w_{2}+h_{\alpha }\cdot \nabla w_{1}+\xi w_{1}\right\Vert
_{\Omega }^{2} \label{specnorm}
\end{eqnarray}
for every $\varphi =\left[ p_{0},u_{0},w_{1},w_{2}\right] \in \mathcal{H}%
_{0}.$ Here, 
\begin{equation}
\textbf{a)}\text{ \ } \xi =\frac{(\frac{1}{2}-C_{2}r_{\mathbf{U}})-\sqrt{(\frac{1}{2}-C_{2}r_{%
\mathbf{U}})^{2}-4C_{2}(C_{1}+C_{2}r_{\mathbf{U}})r_{\mathbf{U}}}}{%
2(C_{1}+C_{2}r_{\mathbf{U}})},\text{ \ \ \ \ \ \textbf{b)}  \ }r_{\mathbf{U}}=\left\Vert 
\mathbf{U}\right\Vert _{\ast }+\left\Vert \mathbf{U}\right\Vert _{\ast
}^{2}+\left\Vert \mathbf{U}\right\Vert _{\ast }^{3}  \label{r-u}
\end{equation}%
where $C_1,C_2>0$ are independent of measurement $\left\Vert 
\mathbf{U}\right\Vert _{\ast }$ which is given as
\begin{equation}
\left\Vert \mathbf{U}\right\Vert _{\ast }=\left\Vert \mathbf{U}\right\Vert
_{L^{\infty }(\mathcal{O})}+\left\Vert \text{div}(\mathbf{U)}\right\Vert
_{L^{\infty }(\mathcal{O})}+\left\Vert \mathbf{U}|_{\Omega }\right\Vert
_{C^{2}(\overline{\Omega })}.\label{u-st}
\end{equation}
Also,

\textbf{(i)} the function $\psi =\psi (f,\chi )\in H^{1}(\mathcal{O)}$ is
considered to solve the following BVP for data $f\in L^{2}(\mathcal{O})$ and 
$\chi \in L^{2}(\Omega )$

\begin{equation}
\left\{ 
\begin{array}{c}
-\Delta \psi =f\text{ \ \ \ in \ }\mathcal{O} \\ 
\frac{\partial \psi }{\partial n}=0\text{ \ \ on \ }S \\ 
\frac{\partial \psi }{\partial n}=\chi \text{ \ \ on \ }\Omega  
\end{array}%
\right. \label{18.1}
\end{equation}%
with the compatibility condition%
\begin{equation}
\int\limits_{\mathcal{O}}fd\mathcal{O+}\int\limits_{\Omega }\chi d\Omega =0.
\end{equation}%
We should note that by known elliptic regularity results for the Neumann
problem on Lipschitz domains--see e.g; \cite{JK}-- we have%
\begin{equation}
\left\Vert \psi (f,\chi )\right\Vert _{H^{\frac{3}{2}}(\mathcal{O)}}\leq %
\left[ \left\Vert f\right\Vert _{\mathcal{O}}+\left\Vert \chi \right\Vert
_{\partial \mathcal{O}}\right] \label{psireg}.
\end{equation}%
\textbf{(ii)} the map $D(\cdot )$ is the Dirichlet map that extends boundary
data $\varphi $ defined on $\Omega $ to a harmonic function in $\mathcal{O}$
satisfying:%
\begin{equation}
D\varphi =f\Leftrightarrow \left\{ 
\begin{array}{c}
\Delta f=0\text{ \ \ in \ }\mathcal{O} \\ 
f|_{\mathcal{\partial O}}=\varphi |_{ext}\text{ \ \ on \ \ }\mathcal{%
\partial O}%
\end{array}%
\right.   \label{20.1}
\end{equation}%
where 
\begin{equation*}
\varphi |_{ext}=%
\begin{cases}
0 & \text{ on }~S \\ 
\phi  & \text{ on }~\Omega 
\end{cases}%
\end{equation*}%
Then by, e.g., \cite[Theorem 3.3.8]{Mc}, and Lax-Milgram Theorem, we deduce
that 
\begin{equation}
D\in \mathcal{L}\big(H_{0}^{1/2+\epsilon }(\Omega );H^{1}(\mathcal{O})\big).
\end{equation}%
\textbf{(iii)} the vector field $h_{\alpha }(\cdot )$ is defined as $%
h_{\alpha }(\cdot )=\mathbf{U}|_{\Omega }-\alpha g,$ where $g(\cdot )$ is a $%
C^{2}$ extension of the normal vector $\mathbf{\nu }(x)$ (recall, with
respect to $\Omega )$ and we specify the parameter $\alpha $ to be 
\begin{equation}\alpha =2\left\Vert \mathbf{U}\right\Vert _{\ast },
\end{equation}%
where $\left\Vert \mathbf{U}\right\Vert _{\ast }$ is as defined in (\ref{u-st}).

\section{Result I: Bounded Semigroup Wellposedness in $\mathcal{H}_0$}

The bounded semigroup wellposedness is given as follows:

\begin{theorem}
\label{wp} \label{wp} With reference to problem (\ref{1})-(\ref{3}), assume
that $\mathbf{U}\in \mathbf{V_{0}}\cap W$ with 
\begin{equation*}
\left\Vert \mathbf{U}\right\Vert _{\ast }=\left\Vert \mathbf{U}\right\Vert
_{L^{\infty }(\mathcal{O})}+\left\Vert \text{div}(\mathbf{U)}\right\Vert
_{L^{\infty }(\mathcal{O})}+\left\Vert \mathbf{U}|_{\Omega }\right\Vert
_{C^{2}(\overline{\Omega })}
\end{equation*}%
sufficiently small. Then we have the following:\newline
\textbf{(i)} The operator $(\mathcal{A}+B):D(\mathcal{A}+B)\cap \mathcal{H}%
_{0}\rightarrow \mathcal{H}_{0}$ is maximal dissipative. In particular, it
obeys the following inequality for all $\phi =\left[ p_{0},{u}%
_{0},w_{1},w_{2}\right] \in D(\mathcal{A}+B)\cap \mathcal{H}_{0}:$%
\begin{equation}
\text{Re}(([\mathcal{A}+B]\varphi ,\varphi ))_{\mathcal{H}_{0}}\leq \left( -%
\frac{1}{4}+C_{\delta }r_{\mathbf{U}}\right) \left[ (\sigma (u_{0}),\epsilon
(u_{0}))_{\mathcal{O}}+\left\Vert u_{0}\right\Vert _{\mathcal{O}}^{2}\right]
+\left(- \frac{1}{2}+\delta C^{\ast }\right) \xi \left[ \left\Vert
p_{0}\right\Vert _{\mathcal{O}}^{2}+\left\Vert \Delta w_{1}\right\Vert
_{\Omega }^{2}\right] ,  \label{ek}
\end{equation}%
where $0<\delta <\frac{1}{2C^{\ast }}$, and $r_{\mathbf{U}}$ is as given in (%
\ref{r-u}).\newline
\textbf{(ii)} Consequently, the operator $(\mathcal{A}+B):D(\mathcal{A}+B)\cap \mathcal{H}%
_{0}\rightarrow \mathcal{H}_{0}$ generates a strongly continuous semigroup $%
\{e^{(\mathcal{A}+B)t}\}_{t\geq 0}$ on $\mathcal{H}_{0}.$ Hence, for every
initial data $\left[ \overline{p},\overline{u},\overline{w_{1}},\overline{%
w_{2}}\right] \in \mathcal{H}_{0},$ the solution $\left[ p(t),{u}%
(t),w_{1}(t),w_{2}(t)\right] $ of problem (\ref{1})-(\ref{3}) is given
continuously by 
\begin{equation}
\left[ 
\begin{array}{c}
p(t) \\ 
u(t) \\ 
w_{1}(t) \\ 
w_{2}(t)%
\end{array}%
\right] =e^{(\mathcal{A}+B)t}\left[ 
\begin{array}{c}
\overline{p} \\ 
\overline{u} \\ 
\overline{w_{1}} \\ 
\overline{w_{2}}%
\end{array}%
\right] \in C([0,T];\mathcal{H}_{0}).
\end{equation}%
Moreover, this semigroup is uniformly bounded in time with respect to the
standard $\mathcal{H}$-inner product. With respect to the special norm in (%
\ref{specnorm}), the semigroup $\{e^{(\mathcal{A}+B)t}\}_{t\geq 0}$ is in fact a contraction.
\end{theorem}

\begin{proof}
Let 
\begin{equation}
\mathcal{A}_{0}\equiv \mathcal{A}+L_{\mathbf{U}}, \label{A-0}
\end{equation}%
where 
\begin{equation}
L_{\mathbf{U}}=\left[ 
\begin{array}{cccc}
0 & 0 & 0 & 0 \\ 
0 & \nabla \mathbf{U}\cdot (\cdot ) &0 & 0 \\ 
0 & 0 & 0 & 0 \\ 
0 & 0 & 0 & 0%
\end{array}%
\right]   \label{l-u}
\end{equation}%
with $D(\mathcal{A}_{0}+B)=D(\mathcal{A}+B).$ Therewith, it is established in
\cite{sima} that $\mathcal{A}_{0}+B$ is maximal dissipative in $\mathcal{H}_{0},$
for all $\left\Vert \mathbf{U}\right\Vert _{\ast }$ sufficiently small. In
particular, we have for $\phi =\left[ p_{0},{u}_{0},w_{1},w_{2}\right] \in D(%
\mathcal{A}_{0}+B)\cap \mathcal{H}_{0}:$%
\begin{equation}
\text{Re}(([\mathcal{A}_{0}+B]\varphi ,\varphi ))_{\mathcal{H}_{0}}\leq -%
\frac{1}{4}(\sigma (u_{0}),\epsilon (u_{0}))_{\mathcal{O}}-\frac{\eta }{4}%
\left\Vert u_{0}\right\Vert _{\mathcal{O}}^{2}-\frac{\xi }{2}\left\Vert
p_{0}\right\Vert _{\mathcal{O}}^{2}-\frac{\xi }{2}\left\Vert \Delta
w_{1}\right\Vert _{\Omega }^{2},  \label{15.2}
\end{equation}%
where parameter $\xi $ is as given in (\ref{r-u}). (For the sake of
completion, the proof of this estimate is given in the Appendix.) Given $%
\phi =\left[ p_{0},{u}_{0},w_{1},w_{2}\right] \in D(\mathcal{A}+B)\cap 
\mathcal{H}_{0},$ invoking (\ref{15.2}) we have%
\begin{eqnarray}
(([\mathcal{A}+B]\varphi ,\varphi ))_{\mathcal{H}_{0}} &=&(([\mathcal{A}%
_{0}+B]\varphi ,\varphi ))_{\mathcal{H}_{0}}-((L_{\mathbf{U}}\varphi
,\varphi ))_{\mathcal{H}_{0}}  \notag \\
&\leq &-\frac{1}{4}(\sigma (u_{0}),\epsilon (u_{0}))_{\mathcal{O}}-\frac{%
\eta }{4}\left\Vert u_{0}\right\Vert _{\mathcal{O}}^{2}-\frac{\xi }{2}%
\left\Vert p_{0}\right\Vert _{\mathcal{O}}^{2}-\frac{\xi }{2}\left\Vert
\Delta w_{1}\right\Vert _{\Omega }^{2}  \notag \\
&&-(\nabla \mathbf{U}\cdot (u_{0}),u_{0}-\alpha D(g\cdot \nabla
w_{1})e_{3}+\xi \nabla \psi (p_{0},w_{1}))_{\mathcal{O}}  \notag \\
&=&-\frac{1}{4}(\sigma (u_{0}),\epsilon (u_{0}))_{\mathcal{O}}-\frac{\eta }{4%
}\left\Vert u_{0}\right\Vert _{\mathcal{O}}^{2}-\frac{\xi }{2}\left\Vert
p_{0}\right\Vert _{\mathcal{O}}^{2}-\frac{\xi }{2}\left\Vert \Delta
w_{1}\right\Vert _{\Omega }^{2}  \notag \\
&&-\int\limits_{\partial \mathcal{O}}(u_{0}\cdot \mathbf{n)U}\cdot \lbrack
u_{0}-\alpha D(g\cdot \nabla w_{1})e_{3}+\xi \nabla \psi
(p_{0},w_{1}))]d\partial \mathcal{O}  \notag \\
&&+\int\limits_{\mathcal{O}}div(u_{0})\mathbf{U}\cdot \lbrack
u_{0}-\alpha D(g\cdot \nabla w_{1})e_{3}+\xi \nabla \psi (p_{0},w_{1}))]d%
\mathcal{O}  \notag \\
&&+\int\limits_{\mathcal{O}}\mathbf{U}\cdot \left( u_{0}\cdot \nabla
\lbrack u_{0}-\alpha D(g\cdot \nabla w_{1})e_{3}+\xi \nabla \psi
(p_{0},w_{1}))]\right) d\mathcal{O}  \label{15.3}
\end{eqnarray}%
We estimate RHS implicitly using the regularity results of \cite{Dauge_1,Dauge_2, Dauge_3},
which are valid under the assumptions made on the geometry.\newline
\textbf{(I)} The mapping in (\ref{20.1}) satisfies%
\begin{equation}
\left\Vert D(g\cdot \nabla w_{1})\right\Vert _{H^{1}(\mathcal{O})}\leq
C\left\Vert w_{1}\right\Vert _{H_{0}^{2}(\Omega )}  \label{15.4}
\end{equation}%
Thus, we have\newline
\textbf{(I-a): }%
\begin{eqnarray}
\left\vert \alpha \int\limits_{\partial \mathcal{O}}(u_{0}\cdot \mathbf{n)U}%
\cdot \overline{D(g\cdot \nabla w_{1})e_{3}}d\partial \mathcal{O}\right\vert
&\leq &\alpha \left\Vert \mathbf{U}\right\Vert _{\ast }\left\Vert
u_{0}\right\Vert _{H^{1}(\mathcal{O})}\left\Vert w_{1}\right\Vert
_{H_{0}^{2}(\Omega )} \notag \\
&=&\sqrt{\left\Vert \mathbf{U}\right\Vert _{\ast }}\left\Vert
u_{0}\right\Vert _{H^{1}(\mathcal{O})}\frac{\alpha \sqrt{\left\Vert \mathbf{U%
}\right\Vert _{\ast }}}{\sqrt{\xi }}\sqrt{\xi }\left\Vert \Delta
w_{1}\right\Vert _{\Omega }  \notag \\
&\leq &C_{\delta }\left\Vert \mathbf{U}\right\Vert _{\ast }\left\Vert
u_{0}\right\Vert _{H^{1}(\mathcal{O})}^{2}+\frac{\delta \alpha
^{2}\left\Vert \mathbf{U}\right\Vert _{\ast }\xi }{\xi }\left\Vert \Delta
w_{1}\right\Vert _{\Omega }^{2}  \notag \\
&\leq &C_{\delta }\left\Vert \mathbf{U}\right\Vert _{\ast }\left\Vert
u_{0}\right\Vert _{H^{1}(\mathcal{O})}^{2}+\delta C\frac{r_{\mathbf{U}}}{%
\xi }\xi \left\Vert \Delta w_{1}\right\Vert _{\Omega }^{2}  \notag \\
&\leq &C_{\delta }\left\Vert \mathbf{U}\right\Vert _{\ast }\left\Vert
u_{0}\right\Vert _{H^{1}(\mathcal{O})}^{2}+\delta C\xi \left\Vert \Delta
w_{1}\right\Vert _{\Omega }^{2}.  \label{15.5}
\end{eqnarray}%
At this point, we should emphasize that $\frac{r_{\mathbf{U}}}{%
\xi }$ will be bounded if $\left\Vert \mathbf{U}\right\Vert _{\ast }$ is small enough and so $C>0$ is independent of $\left\Vert \mathbf{U}\right\Vert _{\ast }$. Also, we implicitly used the Sobolev Trace Theorem for the first
inequality.\newline
\textbf{(I-b): }Again by (\ref{15.4}), similarly%
\begin{equation}
\left\vert \alpha \int\limits_{\mathcal{O}}\mathbf{U}\cdot \left(
u_{0}\cdot \nabla \overline{D(g\cdot \nabla w_{1})e_{3}}\right) d\mathcal{O}%
\right\vert \leq C_{\delta }\left\Vert \mathbf{U}\right\Vert _{\ast
}\left\Vert u_{0}\right\Vert _{H^{1}(\mathcal{O})}^{2}+\delta C\xi
\left\Vert \Delta w_{1}\right\Vert _{\Omega }^{2}  \label{15.55}
\end{equation}%
\newline
\textbf{(II) } Using (\ref{psireg}) and (\ref{15.4}), we have as in (I-a),
\begin{eqnarray}
\left\vert\int\limits_{\mathcal{O}}div(u_{0})\mathbf{U}\cdot \lbrack
u_{0}-\alpha D(g\cdot \nabla w_{1})e_{3}+\xi \nabla \psi (p_{0},w_{1}))]d%
\mathcal{O}\right\vert & \leq & C_{\delta }(\left\Vert \mathbf{U}\right\Vert _{\ast }^{2}+\left\Vert \mathbf{U}\right\Vert _{\ast })\left\Vert
u_{0}\right\Vert _{H^{1}(\mathcal{O})}^{2} \notag\\
&+& \delta C\xi \left[ \left\Vert
p_{0}\right\Vert _{\mathcal{O}}^{2}+\left\Vert \Delta w_{1}\right\Vert
_{\Omega }^{2}\right] \label{f}
\end{eqnarray}

\noindent \textbf{(III) }The mapping (\ref{18.1}) satisfies%
\begin{equation}
\left\vert \psi (p_{0},w_{1})\right\Vert _{H^{2}(\mathcal{O)}}\leq \left[
\left\Vert p_{0}\right\vert _{\mathcal{O}}+\left\Vert w_{1}\right\vert
_{H_{0}^{\frac{1}{2}+\epsilon }(\Omega \mathcal{)}}\right]   \label{15.57}
\end{equation}%
Consequently, we have 
\begin{eqnarray}
&&\left\vert \xi \int\limits_{\partial \mathcal{O}}(u_{0}\cdot \mathbf{n)U}%
\cdot \overline{\nabla \psi (p_{0},w_{1})}d\partial \mathcal{O}\right\vert
+\left\vert \xi \int\limits_{\mathcal{O}}\mathbf{U}\cdot \left( u_{0}\cdot
\nabla \lbrack \nabla \psi (p_{0},w_{1}))]\right) d\mathcal{O}\right\vert  
\notag \\
&\leq &C\xi \left\Vert \mathbf{U}\right\Vert _{\ast }\left\Vert
u_{0}\right\Vert _{H^{1}(\mathcal{O})}\left[ \left\Vert p_{0}\right\Vert _{%
\mathcal{O}}+\left\Vert \Delta w_{1}\right\Vert _{\Omega }\right]   \notag \\
&\leq &C_{\delta }\left\Vert \mathbf{U}\right\Vert _{\ast }^{2}\left\Vert
u_{0}\right\Vert _{H^{1}(\mathcal{O})}^{2}+\delta C\xi \left[ \left\Vert
p_{0}\right\Vert _{\mathcal{O}}^{2}+\left\Vert \Delta w_{1}\right\Vert
_{\Omega }^{2}\right]   \label{15.6}
\end{eqnarray}%
Applying the estimates (\ref{15.5}),(\ref{15.55}), (\ref{f}) and (\ref{15.6}) to RHS of
(\ref{15.3}), we now have 
\begin{eqnarray*}
\text{Re}(([\mathcal{A}+B]\varphi ,\varphi ))_{\mathcal{H}_{0}} &\leq &-%
\frac{1}{4}(\sigma (u_{0}),\epsilon (u_{0}))_{\mathcal{O}}-\frac{\eta }{4}%
\left\Vert u_{0}\right\Vert _{\mathcal{O}}^{2}-\frac{\xi }{2}\left\Vert
p_{0}\right\Vert _{\mathcal{O}}^{2}-\frac{\xi }{2}\left\Vert \Delta
w_{1}\right\Vert _{\Omega }^{2} \\
&&+C_{\delta }\left\Vert \mathbf{U}\right\Vert _{\ast }\left\Vert
u_{0}\right\Vert _{H^{1}(\mathcal{O})}^{2}+C_{\delta }\left\Vert \mathbf{U}%
\right\Vert _{\ast }^{2}\left\Vert u_{0}\right\Vert _{H^{1}(\mathcal{O}%
)}^{2}\\
&&+\delta C\xi \left[ \left\Vert p_{0}\right\Vert _{\mathcal{O}%
}^{2}+\left\Vert \Delta w_{1}\right\Vert _{\Omega }^{2}\right] 
\end{eqnarray*}%
which yields the estimate (\ref{ek}), upon a use of Korn's inequality and this
gives the dissipativity of $\mathcal{A}+B,$ for $\mathbf{U}$ sufficiently
small and $0<\delta <\frac{1}{2C^{\ast }}.$ Moreover, it is shown in \cite{sima}
that $\mathcal{A}_{0}+B$ is maximal dissipative, then by a perturbation
argument--see e.g., p. 211 of \cite{Kesevan}-- $\mathcal{A}+B$ is likewise maximal
dissipative. This concludes the proof of Theorem \ref{wp}.
\end{proof}

\section{Result II: Asymptotic (Strong) Stability}

This section is devoted to address the issue of asymptotic behavior of the
solution whose existence-uniqueness is guaranteed by Theorem \ref{wp}. In
this regard, we show that the system given in (\ref{1})-(\ref{3}) is
strongly stable in $\mathcal{H}_{0}$. That is, given any $\left[ \overline{p}%
,\overline{u},\overline{w_{1}},\overline{w_{2}}\right] \in \mathcal{H}_{0},$
the corresponding solution $\left[ p(t),{u}(t),w_{1}(t),w_{2}(t)\right] \in
C([0,T];\mathcal{H}_{0})$ of (\ref{1})-(\ref{3}) satisfies%
\begin{equation*}
\lim_{t\rightarrow \infty }\left\Vert |\left[ p(t),{u}(t),w_{1}(t),w_{2}(t)%
\right] |\right\Vert _{\mathcal{H}_{0}}=0.
\end{equation*}
Our proof will be based on an ultimate appeal to the pointwise criterion 
\cite[pp. 26, Theorem 8.4 (i)]{CT}:

\begin{theorem}
\label{CT} Let $A$ be the generator of a bounded $C_{0}$-semigroup on a
Hilbert space $X$. If there exists a dense set $M\subset X$ such that 
\begin{equation*}
\underset{\alpha \rightarrow 0^{+}}{\lim }\sqrt{\alpha }R(\alpha +i\beta
;A)x=0 
\end{equation*}%
for every $x\in M$ and every $\beta \in 
\mathbb{R}
,$ then the semigroup is stable.\bigskip
\end{theorem}

\noindent There are reasons why we use here the resolvent criterion in
\cite{CT}, instead of the more wellknown approaches in \cite{Levan} and \cite{arendt}. For
one, the domain of the flow-structure generator $\mathcal{A}+B$ is not
compactly embedded into $\mathcal{H}_{0},$ and so the classical
stabilizability approach in \cite{Levan} is not possible. In addition, the
spectrum criterion for stability in \cite{arendt} is not applicable here since it
would entail the elimination of all three parts of $\sigma (\mathcal{A}+B);$
particularly, the analysis of residual spectrum will be of great technical
complexity and mathematical challenge since it requires one to prove that $%
i\beta $ $(\beta \neq 0)$ is not an eigenvalue of the adjoint of the
generator $\mathcal{A}+B$. In order to justify our concern related to a
residual spectrum analysis and to see how technical calculations the adjoint
of the generator will require, we provide here the adjoint operator $(%
\mathcal{A}+B)^{\ast }$ (For a detailed proof, the reader is referred to 
\cite[Lemma 13]{sima}):

\begin{lemma}
\label{adj} With reference to problem (\ref{1})-(\ref{3}), the adjoint
operator 
\begin{equation*}
(\mathcal{A}+B)^{\ast }:D((\mathcal{A}+B)^{\ast })\cap \mathcal{H}%
_{0}\subset \mathcal{H}_{0}\rightarrow \mathcal{H}_{0}
\end{equation*}%
of the semigroup generator $\mathcal{A}+B$ is defined as%
\begin{equation*}
(\mathcal{A}+B)^{\ast }=\mathcal{A}^{\ast }+B^{\ast }
\end{equation*}%
\begin{equation*}
=\left[ 
\begin{array}{cccc}
\mathbf{U}\mathbb{\cdot }\nabla (\cdot ) & \text{div}(\cdot ) & 0 & 0 \\ 
\mathbb{\nabla (\cdot )} & \text{div}\sigma (\cdot )-\eta I+\mathbf{U}%
\mathbb{\cdot \nabla (\cdot )-\nabla }\mathbf{U\cdot (\cdot )} & 0 & 0 \\ 
0 & 0 & 0 & -I \\ 
-\left[ \cdot \right] _{\Omega } & -\left[ 2\nu \partial _{x_{3}}(\cdot
)_{3}+\lambda \text{div}(\cdot )\right] _{\Omega } & \Delta ^{2} & 0%
\end{array}%
\right] 
\end{equation*}%
\begin{equation*}
+\left[ 
\begin{array}{cccc}
\text{div}(\mathbf{U)}(\cdot ) & 0 & 0 & 0 \\ 
0 & \text{div}(\mathbf{U)}(\cdot ) & 0 & 0 \\ 
{{\mathring{A}}^{-1}}\left\{ \text{div}{{([U}_{1},U_{2}]{)+\mathbf{U\cdot }%
\nabla )}}\right\} {(\cdot )|}_{\Omega } & {{\mathring{A}}^{-1}\left\{ \text{%
div}{{[U}_{1},U_{2}]{+\mathbf{U\cdot }\nabla }}\right\} }\left[ 2\nu
\partial _{x_{3}}(\cdot )_{3}+\lambda \text{div}(\cdot )\right] _{\Omega } & 
0 & 0 \\ 
0 & 0 & 0 & 0%
\end{array}%
\right] 
\end{equation*}%
\begin{equation*}
+\left[ 
\begin{array}{cccc}
-\text{div}(\mathbf{U)(\cdot )} & 0 & 0 & 0 \\ 
0 & 0 & 0 & 0 \\ 
0 & 0 & -{{\mathring{A}}^{-1}\left\{ (\text{div}{{[U}_{1},U_{2}]{+\mathbf{%
U\cdot }\nabla )}\Delta }^{2}(\cdot )\right\} +}\mathbf{U}\mathbb{\cdot }%
\nabla (\cdot )+\Delta {{\mathring{A}}^{-1}\nabla }^{\ast }(\mathbb{\nabla
\cdot }(\mathbf{U}\mathbb{\cdot }\nabla (\cdot ))\mathbb{)} & 0 \\ 
0 & 0 & 0 & 0%
\end{array}%
\right] 
\end{equation*}%
\begin{equation}
=L_{1}+L_{2}+B^{\ast }  \label{adj-AplusB}
\end{equation}%
Here, $\nabla ^{\ast }\in \mathcal{L}(L^{2}(\Omega ),[H^{1}(\Omega
)]^{^{\prime }})$ is the adjoint of the gradient operator $\nabla \in 
\mathcal{L}(H^{1}(\Omega ),L^{2}(\Omega ))$ and the domain of $(\mathcal{A}%
+B)^{\ast }|_{\mathcal{H}_{0}}$ is given as 
\begin{equation*}
D((\mathcal{A}+B)^{\ast })\cap \mathcal{H}_{0}=\{(p_{0},u_{0},w_{1},w_{2})%
\in L^{2}(\mathcal{O})\times \mathbf{H}^{1}(\mathcal{O})\times
H_{0}^{2}(\Omega )\times L^{2}(\Omega )~:~\text{properties }\mathbf{(A^{\ast
}.i)}\text{--}\mathbf{(A^{\ast }.vii)}~~\text{hold}\},
\end{equation*}%
where
\end{lemma}

\begin{enumerate}
\item $\mathbf{(A^{\ast }.i)}$ $\mathbf{U}\cdot \nabla p_{0}\in L^{2}(%
\mathcal{O})$

\item $\mathbf{(A^{\ast }.ii)}$ $\text{div}~\sigma (u_{0})+\nabla p_{0}\in 
\mathbf{L}^{2}(\mathcal{O})$ (So, $\left[ \sigma (u_{0})\mathbf{n}+p_{0}%
\mathbf{n}\right] _{\partial \mathcal{O}}\in \mathbf{H}^{-\frac{1}{2}%
}(\partial \mathcal{O})$)

\item $\mathbf{(A^{\ast }.iii)}$ $\Delta ^{2}w_{1}-\left[ 2\nu \partial
_{x_{3}}(u_{0})_{3}+\lambda \text{div}(u_{0})\right] _{\Omega
}-p_{0}|_{\Omega }\in L^{2}(\Omega )$

\item $\mathbf{(A^{\ast }.iv)}$ $\left( \sigma (u_{0})\mathbf{n}+p_{0}%
\mathbf{n}\right) \bot ~TH^{1/2}(\partial \mathcal{O})$. That is, 
\begin{equation*}
\left\langle \sigma (u_{0})\mathbf{n}+p_{0}\mathbf{n},\mathbf{\tau }%
\right\rangle _{\mathbf{H}^{-\frac{1}{2}}(\partial \mathcal{O})\times 
\mathbf{H}^{\frac{1}{2}}(\partial \mathcal{O})}=0\text{ \ in }\mathcal{D}%
^{\prime }(\mathcal{O})\text{\ for every }\mathbf{\tau }\in
TH^{1/2}(\partial \mathcal{O})
\end{equation*}

\item $\mathbf{(A^{\ast }.v)}$ The flow velocity component $u_{0}=\mathbf{f}%
_{0}+\widetilde{\mathbf{f}}_{0}$, where $\mathbf{f}_{0}\in \mathbf{V}_{0}$
and $\widetilde{\mathbf{f}}_{0}\in \mathbf{H}^{1}(\mathcal{O})$ satisfies 
\begin{equation*}
\widetilde{\mathbf{f}}_{0}=%
\begin{cases}
0 & ~\text{ on }~S \\ 
w_{2}\mathbf{n} & ~\text{ on}~\Omega%
\end{cases}%
\end{equation*}%
\noindent (and so $\left. \mathbf{f}_{0}\right\vert _{\partial \mathcal{O}%
}\in TH^{1/2}(\partial \mathcal{O})$)

\item $\mathbf{(A^{\ast }.vi)}$ $[-w_{2}+\mathbf{U}\mathbb{\cdot }\nabla
w_{1}+\Delta {{\mathring{A}}^{-1}\nabla }^{\ast }(\mathbb{\nabla \cdot }(%
\mathbf{U}\mathbb{\cdot }\nabla w_{1})\mathbb{)}]\in H_{0}^{2}(\Omega ),$ \
(and so $w_{2}\in H_{0}^{1}(\Omega )$)

\item $\mathbf{(A^{\ast }.vii)}$ $\int\limits_{\mathcal{O}}[\mathbf{U}\cdot
\nabla p_{0}+$div$~(u_{0})]d\mathcal{O}$\newline
$+\int\limits_{\Omega }{{\mathring{A}}^{-1}}\left\{ (\text{div}{[U}_{1},U_{2}%
{]+\mathbf{U\cdot }\nabla )(\left[ p_{0}+2\nu \partial _{x_{3}}(u_0
)_{3}+\lambda \text{div}(u_{0})\right] _{\Omega })}\right\} d\Omega $\newline
$-\int\limits_{\Omega}{{\mathring{A}}^{-1}\left\{ (\text{div}{{[U}_{1},U_{2}]%
{+\mathbf{U\cdot }\nabla )}\Delta }^{2}{w}_{1}\right\} }d\Omega $\newline
$+\int\limits_{\Omega}[\mathbf{U}\mathbb{\cdot }\nabla w_{1}+\Delta {{%
\mathring{A}}^{-1}\nabla }^{\ast }(\mathbb{\nabla \cdot }(\mathbf{U}\mathbb{%
\cdot }\nabla w_{1})\mathbb{)]}d\Omega $\newline
$=0.$\newline
\end{enumerate}
\noindent Looking at the expression for $\mathcal{A}^{\ast }+B^{\ast }$ in
(\ref{adj-AplusB}), it seems that showing this adjoint operator has empty null space--necessary for
ruling out residual spectrum of $\mathcal{A}+B$--would be a difficult
exercise. Now, our stability result is as follows:  

\begin{theorem}
\label{St} The bounded $C_{0}-$semigroup $\left\{ e^{(\mathcal{A}%
+B)t}\right\} _{t\geq 0}$ given in Theorem \ref{wp} is strongly stable under
the condition that 
\begin{equation*}
\left\Vert \mathbf{U}\right\Vert _{\ast }=\left\Vert \mathbf{U}\right\Vert
_{L^{\infty }(\mathcal{O})}+\left\Vert \text{div}(\mathbf{U)}\right\Vert
_{L^{\infty }(\mathcal{O})}+\left\Vert \mathbf{U}|_{\Omega }\right\Vert
_{C^{2}(\overline{\Omega })}
\end{equation*}%
is small enough . That is, the solution $\phi (t)=\left[ p(t),{u}%
(t),w_{1}(t),w_{2}(t)\right] $ of the PDE system (\ref{1})-(\ref{3}) tends
asymptotically to the zero state for all initial data $\phi _{0}\in \mathcal{%
H}_{0}.$
\end{theorem}

\begin{proof}
The proof relies on the pointwise resolvent criterion given in Theorem \ref%
{CT}, and hence to obtain the strong stability result it will be enough to
show that the resolvent operator 
\begin{equation*}
R(a+ib;[\mathcal{A}+B])=((a+ib)I-[\mathcal{A}+B])^{-1}
\end{equation*}%
obeys the limit estimate:%
\begin{equation}
\underset{a\rightarrow 0^{+}}{\lim }\left\Vert \left\vert \sqrt{a}R(a+ib;[%
\mathcal{A}+B])\phi ^{\ast }\right\vert \right\Vert _{\mathcal{H}_{0}}=0
\label{R}
\end{equation}%
where $a>0$, $b\in 
\mathbb{R}
$ and given $\phi ^{\ast }\in \mathcal{H}_{0}.$ Here, we invoke the special
inner product $((\cdot ,\cdot ))_{\mathcal{H}_{0}}$ (defined in (\ref{innpr}%
)) to get the necessary estimates but since the norms $||\cdot ||_{\mathcal{H%
}}$ and $||\cdot ||_{\mathcal{H}_{0}}$ are equivalent, we obtain the strong
stability with respect to the standard inner product as well. Let 
\begin{equation*}
\phi =\left[ 
\begin{array}{c}
p_{0} \\ 
u_{0} \\ 
w_{1} \\ 
w_{2}%
\end{array}%
\right] =R(a+ib;[\mathcal{A}+B])\phi ^{\ast }\in D(\mathcal{A}+B)\cap 
\mathcal{H}_{0},\text{ \ \ \ \ \ \ }\phi ^{\ast }=\left[ 
\begin{array}{c}
p_{0}^{\ast } \\ 
u_{0}^{\ast } \\ 
w_{1}^{\ast } \\ 
w_{2}^{\ast }%
\end{array}%
\right] \in \mathcal{H}_{0}
\end{equation*}%
Then $\phi $ solves the following static PDE system:%
\begin{align}
& \left\{ 
\begin{array}{l}
(a+ib)p_{0}+\mathbf{U}\cdot \nabla p_{0}+\text{div}~u_{0}\mathbf{+}\text{div}%
(\mathbf{U)}p_{0}=p_{0}^{\ast }~\text{ in }~\mathcal{O} \\ 
(a+ib)u_{0}+\mathbf{U}\cdot \nabla u_{0}+\nabla \mathbf{U\cdot u}_{0}-\text{%
div}\sigma (u_{0})+\eta u_{0}+\nabla p_{0}=u_{0}^{\ast }~\text{ in }~%
\mathcal{O} \\ 
(\sigma (u_{0})\mathbf{n}-p_{0}\mathbf{n})\cdot \boldsymbol{\tau }=0~\text{
on }~\partial \mathcal{O} \\ 
u_{0}\cdot \mathbf{n}=0~\text{ on }~S \\ 
u_{0}\cdot \mathbf{n}=w_{2}+\mathbf{U}\cdot \nabla w_{1}\text{ \ \ on }%
~\Omega \text{ }%
\end{array}%
\right.   \label{bir} \\
& \left\{ 
\begin{array}{l}
(a+ib)w_{1}-w_{2}-\mathbf{U}\cdot \nabla u_{0}=w_{1}^{\ast } \\ 
(a+ib)w_{1}+\Delta ^{2}w_{1}+\left[ 2\nu \partial
_{x_{3}}(u_{0})_{3}+\lambda \text{div}(u_{0})-p_{0}\right] _{\Omega
}=w_{2}^{\ast }~\text{ on }~\Omega  \\ 
w_{1}=\frac{\partial w_{1}}{\partial \nu }=0~\text{ on }~\partial \Omega 
\end{array}%
\right.   \label{iki}
\end{align}%
We invoke the dissipativity estimate (\ref{ek}) in Theorem \ref{wp} (i):%
\begin{equation*}
(a+ib)\left\Vert \left\vert \phi \right\vert \right\Vert _{\mathcal{H}%
_{0}}^{2}-\left( \left( [\mathcal{A}+B])\phi ,\phi \right) \right) _{%
\mathcal{H}_{0}}=\left( \left( \phi ^{\ast },\phi \right) \right) _{\mathcal{%
H}_{0}},
\end{equation*}%
or%
\begin{equation*}
a\left\Vert \left\vert \phi \right\vert \right\Vert _{\mathcal{H}_{0}}^{2}-%
\text{Re}\left( \left( [\mathcal{A}+B])\phi ,\phi \right) \right) _{\mathcal{%
H}_{0}}=\text{Re}\left( \left( \phi ^{\ast },\phi \right) \right) _{\mathcal{%
H}_{0}}
\end{equation*}%
which gives%
\begin{equation}
\left( \frac{1}{4}-C_{\delta }r_{\mathbf{U}}\right) \left[ (\sigma
(u_{0}),\epsilon (u_{0}))_{\mathcal{O}}+\left\Vert u_{0}\right\Vert _{%
\mathcal{O}}^{2}\right] +\left( \frac{1}{2}-\delta C^{\ast }\right) \xi %
\left[ \left\Vert p_{0}\right\Vert _{\mathcal{O}}^{2}+\left\Vert \Delta
w_{1}\right\Vert _{\Omega }^{2}\right] \leq \left\vert \text{Re}\left(
\left( \phi ^{\ast },\phi \right) \right) _{\mathcal{H}_{0}}\right\vert .
\label{uc}
\end{equation}%
In turn, from the boundary condition $w_{2}=u_{0}\cdot \mathbf{n}-\mathbf{U}%
\cdot \nabla w_{1}$ and (\ref{uc}), we get%
\begin{equation}
\left\Vert w_{2}\right\Vert _{\Omega }\leq C_{\xi ,\mathbf{U}}\sqrt{%
\left\vert \text{Re}\left( \left( \phi ^{\ast },\phi \right) \right) _{%
\mathcal{H}_{0}}\right\vert }  \label{4}
\end{equation}%
where we have also used implicitly the Sobolev Trace Theorem. Now, combining
(\ref{uc}) and (\ref{4}), and using the equivalence of norms $\left\Vert
\left\vert \cdot \right\vert \right\Vert _{\mathcal{H}_{0}}$ and $\left\Vert
\cdot \right\Vert _{\mathcal{H}}$ on $\mathcal{H}_{0},$ we have then 
\begin{equation*}
\left\Vert \left\vert \phi \right\vert \right\Vert _{\mathcal{H}_{0}}\leq
C_{\xi ,\mathbf{U}}\sqrt{\left\vert \text{Re}\left( \left( \phi ^{\ast
},\phi \right) \right) _{\mathcal{H}_{0}}\right\vert }.
\end{equation*}%
Scaling the above inequality by $\sqrt{a},$ followed by Young's Inequality
gives now%
\begin{eqnarray*}
\sqrt{a}\left\Vert \left\vert \phi \right\vert \right\Vert _{\mathcal{H}%
_{0}} &\leq &\sqrt{a}C_{\xi ,\mathbf{U}}\left\Vert \left\vert \phi ^{\ast
}\right\vert \right\Vert _{\mathcal{H}_{0}}^{\frac{1}{2}}\left\Vert
\left\vert \phi \right\vert \right\Vert _{\mathcal{H}_{0}}^{\frac{1}{2}} \\
&\leq &\frac{\sqrt{a}}{2}C_{\xi ,\mathbf{U}}^{2}\left\Vert \left\vert \phi
^{\ast }\right\vert \right\Vert _{\mathcal{H}_{0}}+\frac{\sqrt{a}}{2}%
\left\Vert \left\vert \phi \right\vert \right\Vert _{\mathcal{H}_{0}}.
\end{eqnarray*}%
From the last inequality, we have%
\begin{equation*}
\frac{\sqrt{a}}{2}\left\Vert \left\vert \phi \right\vert \right\Vert _{%
\mathcal{H}_{0}}\leq \frac{\sqrt{a}}{2}C_{\xi ,\mathbf{U}}^{2}\left\Vert
\left\vert \phi ^{\ast }\right\vert \right\Vert _{\mathcal{H}_{0}},
\end{equation*}%
and so 
\begin{equation*}
\underset{a\rightarrow 0^{+}}{\lim }\sqrt{a}\left\Vert \left\vert \phi
\right\vert \right\Vert _{\mathcal{H}_{0}}=0.
\end{equation*}%
This gives the desired limit estimate (\ref{R}) and concludes the proof of
Theorem \ref{St}.
\end{proof}

\section{Appendix}
In this paper, one of the key points used to analyze the stability properties of the PDE system (\ref{1})-(\ref{3}) is the dissipativity estimate (\ref{ek}) in the inherently invariant subspace of the finite energy space. To show this, we appeal to the dissipativity result given in \cite{sima} for an operator which is closely associated our generator $\mathcal{A}+B$. For the readers convenience, we also provide its proof.

\begin{lemma} \label{diss}
Let $\mathcal{A}_0$ be the operator defined in (\ref{A-0}). Then, with reference to problem (\ref{1})-(\ref{3}), the semigroup
generator $(\mathcal{A}_0+B):D(\mathcal{A}_0+B)\cap \mathcal{H}_{0}\subset
\mathcal{H}_{0}\rightarrow \mathcal{H}_{0}$ is dissipative with respect to inner
product $((\cdot ,\cdot ))_{\mathcal{H}_{0}}$ for $\left\Vert \mathbf{U}\right\Vert _{\ast }=\left\Vert \mathbf{U}\right\Vert
_{L^{\infty }(\mathcal{O})}+\left\Vert \text{div}(\mathbf{U)}\right\Vert
_{L^{\infty }(\mathcal{O})}+\left\Vert \mathbf{U}|_{\Omega }\right\Vert
_{C^{2}(\overline{\Omega })}$ small enough. In
particular, for $\varphi =\left[ p_{0},u_{0},w_{1},w_{2}\right] \in D(%
\mathcal{A}_0+B)\cap \mathcal{H}_{0},$ 
\begin{equation}
\text{Re}(([\mathcal{A}_0+B]\varphi ,\varphi ))_{\mathcal{H}_{0}}\leq -\frac{%
(\sigma (u_{0}),\epsilon (u_{0}))_{\mathcal{O}}}{4}-\frac{\eta \left\Vert
u_{0}\right\Vert _{\mathcal{O}}^{2}}{4}-\frac{\xi \left\Vert
p_{0}\right\Vert _{\mathcal{O}}^{2}}{2}-\frac{\xi \left\Vert \Delta
w_{1}\right\Vert _{\Omega }^{2}}{2},  \label{dissest}
\end{equation}%
where 
$\xi$ is specified in (\ref{33}).

\end{lemma}

\begin{proof}
Given $\varphi =\left[ p_{0},u_{0},w_{1},w_{2}\right] \in D(\mathcal{A}_0%
+B)\cap \mathcal{H}_{0},$ we have 
\begin{equation*}
(([\mathcal{A}_0+B]\varphi ,\varphi ))_{\mathcal{H}_{0}}=(-\mathbf{U}\nabla
p_{0}-\text{div}(u_{0})-\text{div}(\mathbf{U})p_{0},p_{0})_{\mathcal{O}}
\end{equation*}%
\begin{equation*}
+(-\nabla p_{0}+\text{div}\sigma (u_{0})-\eta u_{0}-\mathbf{U}\nabla
u_{0},u_{0}-\alpha D(g\cdot \nabla w_{1})e_{3})_{\mathcal{O}}
\end{equation*}%
\begin{equation*}
+(-\nabla p_{0}+\text{div}\sigma (u_{0})-\eta u_{0}-\mathbf{U}\nabla
u_{0},\xi \nabla \psi (p_{0},w_{1}))_{\mathcal{O}}
\end{equation*}%
\begin{equation*}
-\alpha (D(g\cdot \nabla \lbrack w_{2}+\mathbf{U}\nabla
w_{1}])e_{3},u_{0}-\alpha D(g\cdot \nabla w_{1})e_{3}+\xi \nabla \psi
(p_{0},w_{1}))_{\mathcal{O}}
\end{equation*}%
\begin{equation*}
+\xi (\nabla \psi (-\mathbf{U}\nabla p_{0}-\text{div}(u_{0})-\text{div}(%
\mathbf{U})p_{0},w_{2}+\mathbf{U}\nabla w_{1}),u_{0}-\alpha D(g\cdot \nabla
w_{1})e_{3})_{\mathcal{O}}
\end{equation*}%
\begin{equation*}
+\xi ^{2}(\nabla \psi (-\mathbf{U}\nabla p_{0}-\text{div}(u_{0})-\text{div}(%
\mathbf{U})p_{0},w_{2}+\mathbf{U}\nabla w_{1}),\nabla \psi (p_{0},w_{1}))_{%
\mathcal{O}}
\end{equation*}%
\begin{equation*}
+(\Delta w_{2},\Delta w_{1})_{\Omega }+(\Delta (\mathbf{U}\nabla
w_{1}),\Delta w_{1})_{\Omega }
\end{equation*}%
\begin{equation*}
+(p_{0}|_{\Omega }-\left[ 2\nu \partial _{x_{3}}(u_{0})_{3}+\lambda \text{div%
}(u_{0})\right] |_{\Omega },w_{2}+h_{\alpha }\cdot \nabla w_{1}+\xi
w_{1})_{\Omega }
\end{equation*}%
\begin{equation*}
+(h_{\alpha }\cdot \nabla \lbrack w_{2}+\mathbf{U}\nabla
w_{1}],w_{2}+h_{\alpha }\cdot \nabla w_{1}+\xi w_{1})_{\Omega }
\end{equation*}%
\begin{equation*}
-(\Delta ^{2}w_{1},w_{2}+h_{\alpha }\cdot \nabla w_{1}+\xi w_{1})_{\Omega }
\end{equation*}%
\begin{equation*}
+\xi (w_{2}+\mathbf{U}\nabla w_{1},w_{2}+h_{\alpha }\cdot \nabla w_{1}+\xi
w_{1})_{\Omega }.
\end{equation*}%
After integration by parts we then arrive at%
\begin{equation*}
(([\mathcal{A}_0+B]\varphi ,\varphi ))_{\mathcal{H}_{0}}=-(\sigma
(u_{0}),\epsilon (u_{0}))_{\mathcal{O}}-\eta \left\Vert u_{0}\right\Vert _{%
\mathcal{O}}^{2}+\frac{1}{2}\int\limits_{\mathcal{O}}\text{div}(\mathbf{U}%
)[|u_{0}|^{2}-|p_{0}|^{2}]d\mathcal{O}
\end{equation*}%
\begin{equation*}
+2i\text{Im}[(p_{0},\text{div}(u_{0}))_{\mathcal{O}}+(\Delta w_{2},\Delta
w_{1})_{\Omega }]-i\text{Im}[(\mathbf{U}\nabla p_{0},p_{0})_{\mathcal{O}}+(%
\mathbf{U}\nabla u_{0},u_{0})_{\mathcal{O}}]
\end{equation*}%
\begin{equation}
+\sum\limits_{j=1}^{8}I_{j},  \label{est}
\end{equation}%
where above the $I_{j}$ are given by:%
\begin{equation*}
I_{1}=(\nabla p_{0}-\text{div}\sigma (u_{0})+\eta u_{0}+\mathbf{U}\nabla
u_{0},\alpha D(g\cdot \nabla w_{1})e_{3})_{\mathcal{O}}
\end{equation*}%
\begin{equation}
-\alpha (p_{0}|_{\Omega }-\left[ 2\nu \partial _{x_{3}}(u_{0})_{3}+\lambda 
\text{div}(u_{0})\right] |_{\Omega },g\cdot \nabla w_{1})_{\Omega },
\label{I1}
\end{equation}%
\begin{equation*}
I_{2}=(-\nabla p_{0}+\text{div}\sigma (u_{0})-\eta u_{0}-\mathbf{U}\nabla
u_{0},\xi \nabla \psi (p_{0},w_{1}))_{\mathcal{O}}-\xi (\Delta
^{2}w_{1},w_{1})_{\Omega }
\end{equation*}%
\begin{equation}
+(p_{0}|_{\Omega }-\left[ 2\nu \partial _{x_{3}}(u_{0})_{3}+\lambda \text{div%
}(u_{0})\right] |_{\Omega },\xi w_{1})_{\Omega },  \label{I2}
\end{equation}%
\begin{equation}
I_{3}=-\alpha (D(g\cdot \nabla \lbrack w_{2}+\mathbf{U}\nabla
w_{1}])e_{3},u_{0}-\alpha D(g\cdot \nabla w_{1})e_{3}+\xi \nabla \psi
(p_{0},w_{1}))_{\mathcal{O}},  \label{I3}
\end{equation}%
\begin{equation}
I_{4}=\xi (\nabla \psi (-\mathbf{U}\nabla p_{0}-\text{div}(u_{0})-\text{div}(%
\mathbf{U})p_{0},w_{2}+\mathbf{U}\nabla w_{1}),u_{0}-\alpha D(g\cdot \nabla
w_{1})e_{3})_{\mathcal{O}},  \label{I4}
\end{equation}%
\begin{equation}
I_{5}=\xi ^{2}(\nabla \psi (-\mathbf{U}\nabla p_{0}-\text{div}(u_{0})-\text{%
div}(\mathbf{U})p_{0},w_{2}+\mathbf{U}\nabla w_{1}),\nabla \psi
(p_{0},w_{1}))_{\mathcal{O}},  \label{I5}
\end{equation}%
\begin{equation}
I_{6}=(\Delta (\mathbf{U}\nabla w_{1}),\Delta w_{1})_{\Omega }-(\Delta
^{2}w_{1},h_{\alpha }\cdot \nabla w_{1})_{\Omega },  \label{I6}
\end{equation}%
\begin{equation}
I_{7}=(h_{\alpha }\cdot \nabla \lbrack w_{2}+\mathbf{U}\nabla
w_{1}],w_{2})_{\Omega },  \label{I7}
\end{equation}%
\begin{equation*}
I_{8}=(h_{\alpha }\cdot \nabla \lbrack w_{2}+\mathbf{U}\nabla
w_{1}],h_{\alpha }\cdot \nabla w_{1}+\xi w_{1})_{\Omega }
\end{equation*}%
\begin{equation}
+\xi (w_{2}+\mathbf{U}\nabla w_{1},w_{2}+h_{\alpha }\cdot \nabla w_{1}+\xi
w_{1})_{\Omega }.  \label{I8}
\end{equation}%
where we also recall the definition $h_{\alpha}=\mathbf{U}|_{\Omega}-\alpha
g.$ In the course of estimating the terms (\ref{I1})-(\ref{I8}) above, we
will invoke the polynomial 
\begin{equation*}
r(a)=a+a^{2}+a^{3}, 
\end{equation*}%
and for the simplicity, we set 
\begin{equation*}
r_{\mathbf{U}}=r(\left\Vert 
\mathbf{U}\right\Vert _{\ast }). \label{rU}
\end{equation*}
We start with $I_{1};$ integrating by parts, we have%
\begin{equation*}
I_{1}=-\alpha (p_{0},\text{div}[D(g\cdot \nabla w_{1})e_{3}])_{\mathcal{O}%
}+\alpha (\sigma (u_{0}),\epsilon (D(g\cdot \nabla w_{1})e_{3})_{\mathcal{O}}
\end{equation*}%
\begin{equation}
+\alpha \eta (u_{0},D(g\cdot \nabla w_{1})e_{3})_{\mathcal{O}}+\alpha (%
\mathbf{U}\nabla u_{0},D(g\cdot \nabla w_{1})e_{3})_{\mathcal{O}}
\label{I1-1}
\end{equation}%
Using the fact that Dirichlet map $D\in L(H_{0}^{\frac{1}{2}+\epsilon
}(\Omega ),H^{1}(\mathcal{O}))$, we have%
\begin{equation}
I_{1}\leq r_{\mathbf{U}}C\left\{ \left\Vert
u_{0}\right\Vert _{H^{1}(\mathcal{O})}^{2}+\left\Vert p_{0}\right\Vert _{%
\mathcal{O}}^{2}+\left\Vert \Delta w_{1}\right\Vert _{\Omega }^{2}\right\}
\label{A}
\end{equation}%
%
%
%
We continue with $I_{2};$ using the definition of the map $\psi (\cdot
,\cdot )$ in (\ref{18.1}) and integrating by parts we get%
\begin{equation*}
I_{2}=-\xi \int\limits_{\mathcal{O}}\left\vert p_{0}\right\vert ^{2}d%
\mathcal{O}-\xi (\sigma (u_{0}),\epsilon (\nabla \psi (p_{0},w_{1})))_{%
\mathcal{O}}
\end{equation*}%
\begin{equation*}
+\xi \left\langle \sigma (u_{0})\mathbf{n}-p_{0}\mathbf{n},(\nabla \psi
(p_{0},w_{1}),\mathbf{n})\mathbf{n}\right\rangle _{\partial \mathcal{O}}-\eta (u_{0},\xi
\nabla \psi (p_{0},w_{1}))_{\mathcal{O}}
\end{equation*}%
\begin{equation*}
(-\mathbf{U}\nabla u_{0},\xi \nabla \psi (p_{0},w_{1}))_{\mathcal{O}%
}-(\Delta ^{2}w_{1},\xi w_{1})_{\Omega }
\end{equation*}%
\begin{equation*}
+(p_{0}|_{\Omega }-\left[ 2\nu \partial _{x_{3}}(u_{0})_{3}+\lambda \text{div%
}(u_{0})\right] |_{\Omega },\xi w_{1})_{\Omega },
\end{equation*}%
whence we obtain%
\begin{equation*}
I_{2}\leq -\xi \left\Vert p_{0}\right\Vert _{\mathcal{O}}^{2}-\xi \left\Vert
\Delta w_{1}\right\Vert _{\Omega }^{2}+\xi r_{\mathbf{U}}C\left\{ \left\Vert u_{0}\right\Vert _{H^{1}(\mathcal{O}%
)}^{2}+\left\Vert p_{0}\right\Vert _{\mathcal{O}}^{2}+\left\Vert \Delta
w_{1}\right\Vert _{\Omega }^{2}\right\}
\end{equation*}%
\begin{equation}
+\xi C\left\{ \left\Vert u_{0}\right\Vert _{H^{1}(\mathcal{O})}\left[
\left\Vert p_{0}\right\Vert _{\mathcal{O}}+\left\Vert \Delta
w_{1}\right\Vert _{\Omega }\right] \right\} .  \label{B}
\end{equation}%
For $I_{3}:$ recalling the boundary condition 
\begin{equation*}
(u_{0})_{3}|_{\Omega }=w_{2}+\mathbf{U}\nabla w_{1},
\end{equation*}%
making use of Lemma 6.1 of \cite{material} and considering the assumptions
made on the geometry, we have%
\begin{equation*}
I_{3}\leq \alpha C\left\Vert g\cdot \nabla (u_{0})_{3}\right\Vert _{H^{-%
\frac{1}{2}}(\Omega )}\left\Vert u_{0}-\alpha D(g\cdot \nabla
w_{1})e_{3}+\xi \nabla \psi (p_{0},w_{1})\right\Vert _{\mathcal{O}}
\end{equation*}%
\begin{equation}
\leq C\left[ r_{\mathbf{U}}\left\{ \left\Vert
u_{0}\right\Vert _{H^{1}(\mathcal{O})}^{2}+\left\Vert \Delta
w_{1}\right\Vert _{\Omega }^{2}\right\} +\xi ^{2}\left\{ \left\Vert
p_{0}\right\Vert _{\mathcal{O}}^{2}+\left\Vert \Delta w_{1}\right\Vert
_{\Omega }^{2}\right\} \right]  \label{C}
\end{equation}%
where we have also implicitly used the Sobolev Embedding Theorem. To
continue with $I_{4}:$%
\begin{equation*}
I_{4}=\xi (\nabla \psi (-\mathbf{U}\nabla p_{0}-\text{div}(\mathbf{U}%
)p_{0},0),u_{0}-\alpha D(g\cdot \nabla w_{1})e_{3})_{\mathcal{O}}
\end{equation*}%
\begin{equation*}
+\xi (\nabla \psi (-\text{div}(u_{0}),u_{0}\cdot \mathbf{n}),u_{0}-\alpha
D(g\cdot \nabla w_{1})e_{3})_{\mathcal{O}}
\end{equation*}%
\begin{equation}
=I_{4a}+I_{4b}  \label{D1}
\end{equation}%
Since $\mathbf{U}\cdot \mathbf{n|}_{\partial \mathcal{O}}\mathbf{=0,}$ we
have that $(\mathbf{U}\nabla p_{0}+$div$(\mathbf{U})p_{0})\in \lbrack H^{1}(%
\mathcal{O})]^{^{\prime }}$ with%
\begin{equation}
\left\Vert \mathbf{U}\nabla p_{0}+\text{div}(\mathbf{U})p_{0}\right\Vert
_{[H^{1}(\mathcal{O})]^{^{\prime }}}\leq C\left\Vert \mathbf{U}\right\Vert
_{\ast }\left\Vert p_{0}\right\Vert _{\mathcal{O}}.  \label{D*}
\end{equation}%
By Lax-Milgram Theorem, we then have%
\begin{equation*}
I_{4a}\leq C\xi \left\Vert \nabla \psi (-\mathbf{U}\nabla p_{0}-\text{div}(%
\mathbf{U})p_{0},0)\right\Vert _{\mathcal{O}}\left\Vert u_{0}-\alpha
D(g\cdot \nabla w_{1})e_{3}\right\Vert _{\mathcal{O}}
\end{equation*}%
\begin{equation}
\leq C\xi r_{\mathbf{U}}\left\{ \left\Vert
u_{0}\right\Vert _{H^{1}(\mathcal{O})}^{2}+\left\Vert p_{0}\right\Vert _{%
\mathcal{O}}^{2}+\left\Vert \Delta w_{1}\right\Vert _{\Omega }^{2}\right\}
\label{D2}
\end{equation}%
and similarly%
\begin{equation}
I_{4b}\leq C\xi r_{\mathbf{U}}\left\{
\left\Vert u_{0}\right\Vert _{H^{1}(\mathcal{O})}^{2}+\left\Vert \Delta
w_{1}\right\Vert _{\Omega }^{2}\right\} .  \label{D3}
\end{equation}%
Now, applying (\ref{D2})-(\ref{D3}) to (\ref{D1}) gives%
\begin{equation}
I_{4}\leq C\xi r_{\mathbf{U}}\left\{
\left\Vert u_{0}\right\Vert _{H^{1}(\mathcal{O})}^{2}+\left\Vert
p_{0}\right\Vert _{\mathcal{O}}^{2}+\left\Vert \Delta w_{1}\right\Vert
_{\Omega }^{2}\right\} .  \label{D}
\end{equation}%
Estimating $I_{5}:$ we proceed as before done for $I_{4}$ and invoke (\ref%
{D*}), Lax Milgram Theorem and the estimate (\ref{psireg}) to have%
\begin{equation}
I_{5}\leq C\xi ^{2}\left[ \left\Vert \mathbf{U}\right\Vert _{\ast }\left\{
\left\Vert p_{0}\right\Vert _{\mathcal{O}}^{2}+\left\Vert \Delta
w_{1}\right\Vert _{\Omega }^{2}\right\} +\left\Vert u_{0}\right\Vert _{H^{1}(%
\mathcal{O})}^{2}\right]  \label{E}
\end{equation}%
For $I_{6},$ in order to estimate the second term in (\ref{I6}), we follow
the standard calculations used for the flux multipliers and the commutator
symbol given by%
\begin{equation}
\lbrack P,Q]f=P(Qf)-Q(Pf)  \label{F*}
\end{equation}%
for the differential operators $P$ and $Q$. Hence,%
\begin{align}
-(\Delta ^{2}w_{1},h_{\alpha}\cdot \nabla w_{1})_{\Omega }=& ~(\nabla \Delta
w_{1},\nabla (h_{\alpha}\cdot \nabla w_{1}))_{\Omega } \\
=& ~-(\Delta w_{1},\Delta (h_{\alpha}\cdot \nabla w_{1}))_{\Omega
}+\int_{\partial \Omega }(h_{\alpha}\cdot \mathbf{\nu })|\Delta
w_{1}|^{2}d\partial \Omega ,
\end{align}%
where, in the first identity we have directly invoked the clamped plate
boundary conditions, and in the second we have used the fact that $%
w_{1}=\partial _{\mathbf{\nu }}w_{1}=0$ on $\partial \Omega $ which yields
that 
\begin{equation*}
\frac{\partial }{\partial \mathbf{\nu }}(h_{\alpha}\cdot \nabla
w_{1})=(h_{\alpha}\cdot \mathbf{\nu })\frac{\partial ^{2}w_{1}}{\partial 
\mathbf{\nu }}=(h_{\alpha}\cdot \mathbf{\nu })(\Delta w_{1}\big|_{\partial
\Omega }).
\end{equation*}
\noindent (See \cite{lagnese} or \cite[p.305]{LT}). Using the commutator
bracket $[\cdot ,\cdot ]$, we can rewrite the latter relation as 
\begin{equation*}
-(\Delta ^{2}w_{1},h_{\alpha}\cdot \nabla w_{1})_{\Omega }=~-(\Delta
w_{1},[\Delta ,h_{\alpha}\cdot \nabla ]w_{1})_{\Omega }-(\Delta
w_{1},h_{\alpha}\cdot \nabla (\Delta w_{1}))_{\Omega }+\int_{\partial \Omega
}(h_{\alpha}\cdot \mathbf{\nu })|\Delta w_{1}|^{2}d\partial \Omega .
\end{equation*}%
With Green's relations once more: 
\begin{align}
-(\Delta ^{2}w_{1},h_{\alpha}\cdot \nabla w_{1})_{\Omega }=& ~-(\Delta
w_{1},[\Delta ,h_{\alpha}\cdot \nabla ]w_{1})_{\Omega }-\frac{1}{2}%
\int_{\partial \Omega }(h_{\alpha}\cdot \mathbf{\nu })|\Delta
w_{1}|^{2}d\partial \Omega  \notag \\
& +\frac{1}{2}\int_{\Omega }\big[\text{div}(h_{\alpha})\big]|\Delta
w_{1}|^{2}d\Omega -i\text{Im}(\Delta w_{1},h_{\alpha}\cdot \nabla (\Delta
w_{1}))_{\Omega }  \notag \\
& +\int_{\partial \Omega }(h_{\alpha}\cdot \mathbf{\nu })|\Delta
w_{1}|^{2}d\partial \Omega .  \label{use1}
\end{align}%
Thus, 
\begin{align}
-(\Delta ^{2}w_{1},h_{\alpha}\cdot \nabla w_{1})_{\Omega }=& ~-(\Delta
w_{1},[\Delta ,h_{\alpha}\cdot \nabla ]w_{1})_{\Omega }+\frac{1}{2}%
\int_{\partial \Omega }(h_{\alpha}\cdot \mathbf{\nu })|\Delta
w_{1}|^{2}d\partial \Omega  \notag \\
& +\frac{1}{2}\int_{\Omega }\big[\text{div}(h_{\alpha})\big]|\Delta
w_{1}|^{2}d\Omega -i\text{Im}(\Delta w_{1},h_{\alpha}\cdot \nabla (\Delta
w_{1})).  \label{use2}
\end{align}%
Since $h_{\alpha}=\mathbf{U}\big|_{\Omega }-\alpha g$, where $g$ is an
extension of $\mathbf{\nu }(\mathbf{x})$, we will have then%
\begin{equation}
-\text{Re}(\Delta ^{2}w_{1},h_{\alpha}\cdot \nabla w_{1})_{\Omega }=~\dfrac{1%
}{2}\int_{\partial \Omega }(\mathbf{U}\cdot \mathbf{\nu }-\alpha )|\Delta
w_{1}|^{2}d\partial \Omega +\dfrac{1}{2}\int_{\Omega }\text{div}%
(h_{\alpha})|\Delta w_{1}|^{2}d\Omega -\text{Re}(\Delta w_{1},[\Delta
,h_{\alpha}\cdot \nabla ]w_{1})_{\Omega }  \label{com1}
\end{equation}%
Since we can explicitly compute the commutator 
\begin{align*}
\lbrack \Delta ,{h_{\alpha}}\cdot \nabla ]w_{1}=& (\Delta h_{1})(\partial
_{x_{1}}w_{1})+2(\partial _{x_{1}}h_{1})(\partial
_{x_{1}}^{2}w_{1})+2(\partial _{x_{2}}h_{2})(\partial
_{x_{2}}^{2}w_{1})+(\Delta h_{2})(\partial _{x_{2}}w_{1}) \\
& +2\text{div}(h_{\alpha})(\partial _{x_{1}}\partial _{x_{2}}w_{1}),
\end{align*}%
and 
\begin{equation}
\big|\big|\lbrack \Delta ,{h_{\alpha}}\cdot \nabla ]w_{1}\big|\big|%
_{L^{2}(\Omega )}\leq r_{\mathbf{U}}||\Delta
w_{1}||_{L^{2}(\Omega )}.  \label{commest}
\end{equation}%
combining (\ref{com1})-(\ref{commest}) we eventually get%
\begin{equation}
-\text{Re}(\Delta ^{2}w_{1},h_{\alpha }\cdot \nabla w_{1})_{\Omega }\leq 
\frac{1}{2}\int\limits_{\partial \Omega }[\mathbf{U\cdot \nu -}\alpha
]\left\vert \Delta w_{1}\right\vert ^{2}d\partial \Omega +Cr_{\mathbf{U}}\left\Vert \Delta w_{1}\right\Vert _{\Omega
}^{2}.  \label{F1}
\end{equation}%
Moreover, for the first term of (\ref{I6}), we have%
\begin{equation*}
(\Delta (\mathbf{U}\nabla w_{1}),\Delta w_{1})_{\Omega }=(\mathbf{U}\nabla
w_{1}),\Delta w_{1})_{\Omega }-([\mathbf{U\cdot }\nabla ,\Delta
]w_{1},\Delta w_{1})_{\Omega }
\end{equation*}%
\begin{equation*}
=\int\limits_{\partial \Omega }(\mathbf{U\cdot \nu })\left\vert \Delta
w_{1}\right\vert ^{2}d\partial \Omega -\int\limits_{\partial \Omega }\text{%
div}(\mathbf{U)}\left\vert \Delta w_{1}\right\vert ^{2}d\partial \Omega
\end{equation*}%
\begin{equation*}
-([\mathbf{U\cdot }\nabla ,\Delta ]w_{1},\Delta w_{1})_{\Omega
}-\int\limits_{\Omega }\Delta w_{1}\mathbf{U\cdot }\nabla (\Delta
w_{1})d\Omega
\end{equation*}%
where we also use the commutator expression in (\ref{F*}). This gives us%
\begin{equation}
\text{Re}(\Delta (\mathbf{U}\nabla w_{1}),\Delta w_{1})_{\Omega }\leq \frac{1%
}{2}\int\limits_{\partial \Omega }(\mathbf{U\cdot \nu })\left\vert \Delta
w_{1}\right\vert ^{2}d\partial \Omega +Cr_{\mathbf{U}}\left\Vert \Delta w_{1}\right\Vert _{\Omega }^{2}.  \label{F2}
\end{equation}%
Now applying (\ref{F1})-(\ref{F2}) to (\ref{I6}), we obtain%
\begin{equation}
\text{Re}I_{6}\leq \int\limits_{\partial \Omega }[\mathbf{U\cdot \nu -}\frac{%
\alpha }{2}]\left\vert \Delta w_{1}\right\vert ^{2}d\partial \Omega
+Cr_{\mathbf{U}}\left\Vert \Delta
w_{1}\right\Vert _{\Omega }^{2}.  \tag{F}
\end{equation}%
To estimate $I_{7}:$ since $w_{2}\in H_{0}^{1}(\Omega ),$ we have 
\begin{equation*}
\text{Re}(h_{\alpha }\cdot \nabla w_{2},w_{2})_{\Omega }=-\frac{1}{2}%
\int\limits_{\Omega }\text{div}(h_{\alpha }\mathbf{)}\left\vert
w_{2}\right\vert ^{2}d\Omega
\end{equation*}%
\begin{equation*}
=-\frac{1}{2}\int\limits_{\Omega }\text{div}(h_{\alpha }\mathbf{)}\left\vert
(u_{0})_{3}-\mathbf{U}\nabla w_{1}\right\vert ^{2}d\Omega
\end{equation*}%
after using the boundary condition in $\mathbf{(A.v)}.$ Applying the last
relation to RHS of (\ref{I7}) and recalling that $h_{\alpha }=\mathbf{U|}%
_{\Omega }-\alpha g,$ we get%
\begin{equation*}
\text{Re}I_{7}=\text{Re}(h_{\alpha }\cdot \nabla w_{2},w_{2})_{\Omega }+%
\text{Re}(h_{\alpha }\cdot \nabla (\mathbf{U}\nabla w_{1}),(u_{0})_{3}-%
\mathbf{U}\nabla w_{1})_{\mathcal{O}}
\end{equation*}%
\begin{equation}
\leq Cr_{\mathbf{U}}\left\{ \left\Vert
u_{0}\right\Vert _{H^{1}(\mathcal{O})}^{2}+\left\Vert \Delta
w_{1}\right\Vert _{\Omega }^{2}\right\}  \label{G}
\end{equation}%
where we also implicitly use Sobolev Trace Theorem. Lastly, for the term $%
I_{8}$, we proceed in a manner similar to that adopted for $I_{7}$ and we
have%
\begin{equation*}
I_{8}=(h_{\alpha }\cdot \nabla (u_{0})_{3},h_{\alpha }\cdot \nabla w_{1}+\xi
w_{1})_{\Omega }
\end{equation*}%
\begin{equation*}
+\xi ((u_{0})_{3},(u_{0})_{3}-\mathbf{U}\cdot \nabla w_{1}+h_{\alpha }\cdot
\nabla w_{1}+\xi w_{1})_{\Omega }
\end{equation*}%
\begin{equation*}
\leq C\left[ r_{\mathbf{U}}+\xi ^{2}\right]
\left\{ \left\Vert u_{0}\right\Vert _{H^{1}(\mathcal{O})}^{2}+\left\Vert
\Delta w_{1}\right\Vert _{\Omega }^{2}\right\}
\end{equation*}%
\begin{equation}
+C\xi \left[ \left\Vert u_{0}\right\Vert _{H^{1}(\mathcal{O}%
)}^{2}+r_{\mathbf{U}}\left\{ \left\Vert
u_{0}\right\Vert _{H^{1}(\mathcal{O})}^{2}+\left\Vert \Delta
w_{1}\right\Vert _{\Omega }^{2}\right\} \right]  \label{H}
\end{equation}%
Now, if we apply (\ref{A})-(\ref{H}) to RHS of (\ref{est}), we obtain%
\begin{equation*}
\text{Re}(([\mathcal{A}_0+B]\varphi ,\varphi ))_{\mathcal{H}_{0}}\leq -(\sigma
(u_{0}),\epsilon (u_{0}))_{\mathcal{O}}-\eta \left\Vert u_{0}\right\Vert _{%
\mathcal{O}}^{2}-\xi \left\Vert p_{0}\right\Vert _{\mathcal{O}}^{2}-\xi
\left\Vert \Delta w_{1}\right\Vert _{\Omega }^{2}
\end{equation*}%
\begin{equation*}
+\int\limits_{\partial \Omega }[\mathbf{U\cdot \nu -}\frac{\alpha }{2}%
]\left\vert \Delta w_{1}\right\vert ^{2}d\partial \Omega
\end{equation*}%
\begin{equation*}
+C\left[ r_{\mathbf{U}}+\xi r_{\mathbf{U}}+\xi ^{2}+\xi \right] \left\Vert
u_{0}\right\Vert _{H^{1}(\mathcal{O})}^{2}
\end{equation*}%
\begin{equation*}
+C\left[ r_{\mathbf{U}}+\xi r_{\mathbf{U}}+\xi ^{2}+\xi ^{2}r_{\mathbf{U}}%
\right] \left\{ \left\Vert p_{0}\right\Vert _{\mathcal{O}}^{2}+\left\Vert
\Delta w_{1}\right\Vert _{\Omega }^{2}\right\}
\end{equation*}%
\begin{equation}
+C\xi \left\Vert u_{0}\right\Vert _{H^{1}(\mathcal{O})}^{2}\left\{
\left\Vert p_{0}\right\Vert _{\mathcal{O}}+\left\Vert \Delta
w_{1}\right\Vert _{\Omega }\right\}.  \label{31}
\end{equation}%
We recall now the value of $\alpha
=2\left\Vert \mathbf{U}\right\Vert _{\ast }$ to get%
\begin{equation*}
\text{Re}(([\mathcal{A}_0+B]\varphi ,\varphi ))_{H_N^{\bot}}\leq -(\sigma
(u_{0}),\epsilon (u_{0}))_{\mathcal{O}}-\eta \left\Vert u_{0}\right\Vert _{%
\mathcal{O}}^{2}-\xi \left\Vert p_{0}\right\Vert _{\mathcal{O}}^{2}-\xi
\left\Vert \Delta w_{1}\right\Vert _{\Omega }^{2}
\end{equation*}%
\begin{equation*}
+\left[ (C_{1}+C_{2}r_{\mathbf{U}})\xi ^{2}+C_{2}r_{\mathbf{U}}\xi +C_{2}r_{%
\mathbf{U}}\right] \left\{ \left\Vert p_{0}\right\Vert _{\mathcal{O}%
}^{2}+\left\Vert \Delta w_{1}\right\Vert _{\Omega }^{2}\right\}
\end{equation*}%
\begin{equation}
+\frac{1}{2}\left\{ (\sigma (u_{0}),\epsilon (u_{0}))_{\mathcal{O}}+\eta
\left\Vert u_{0}\right\Vert _{\mathcal{O}}^{2}\right\}+C_{3}\left[ r_{\mathbf{U}}+\xi r_{\mathbf{U}}+\xi ^{2}+\xi \right]
\left\Vert u_{0}\right\Vert _{H^{1}(\mathcal{O})}^{2}  \label{32}
\end{equation}%
where the positive constants $C_{1},C_{2}$ and $C_{3}$ are obtained with the
application of Holder-Young and Korn's inequalities and $C_{2}$ depends on
the constant in Korn's inequality. We now specify $\xi $ be a zero of the
equation%
\begin{equation*}
(C_{1}+C_{2}r_{\mathbf{U}})\xi ^{2}+(C_{2}r_{\mathbf{U}}-\frac{1}{2})\xi
+C_{2}r_{\mathbf{U}}=0.
\end{equation*}%
Namely, 
\begin{equation}
\xi =\frac{\frac{1}{2}-C_{2}r_{\mathbf{U}}}{2(C_{1}+C_{2}r_{\mathbf{U}})}-%
\frac{\sqrt{(\frac{1}{2}-C_{2}r_{\mathbf{U}})^{2}-4C_{2}(C_{1}+C_{2}r_{%
\mathbf{U}})r_{\mathbf{U}}}}{2(C_{1}+C_{2}r_{\mathbf{U}})}  \label{33}
\end{equation}%
where the radicand is nonnegative for $\left\Vert \mathbf{U}\right\Vert
_{\ast }$ sufficiently small. Then (\ref{32}) becomes%
\begin{equation*}
\text{Re}(([\mathcal{A}_0+B]\varphi ,\varphi ))_{\mathcal{H}_{0}}\leq -\frac{%
(\sigma (u_{0}),\epsilon (u_{0}))_{\mathcal{O}}}{4}-\eta \frac{\left\Vert
u_{0}\right\Vert _{\mathcal{O}}^{2}}{4}-\frac{\xi }{2}\left\Vert
p_{0}\right\Vert _{\mathcal{O}}^{2}-\frac{\xi }{2}\left\Vert \Delta
w_{1}\right\Vert _{\Omega }^{2}
\end{equation*}%
\begin{equation*}
-\frac{(\sigma (u_{0}),\epsilon (u_{0}))_{\mathcal{O}}}{4}-\eta \frac{%
\left\Vert u_{0}\right\Vert _{\mathcal{O}}^{2}}{4}
\end{equation*}%
\begin{equation*}
+C_{K}\left[ r_{\mathbf{U}}+\xi r_{\mathbf{U}}+\xi ^{2}+\xi \right] \left\{
(\sigma (u_{0}),\epsilon (u_{0}))_{\mathcal{O}}+\eta \left\Vert
u_{0}\right\Vert _{\mathcal{O}}^{2}\right\} .
\end{equation*}%
With $\xi $ as prescribed in (\ref{33}), we now have the dissipativity
estimate (\ref{dissest}), for $\left\Vert \mathbf{U}\right\Vert _{\ast }$
small enough. (Here we also implicitly re-use Korn's inequality and $C_{K}$
is the constant there). This concludes the proof of Lemma \ref{diss}.
\end{proof}

\section{Acknowledgement}

\noindent The author would like to thank the National Science Foundation,
and acknowledge her partial funding from NSF Grant DMS-1907823.

\end{document}